\documentclass[11pt]{amsart} 
\usepackage{amsthm,amsbsy,amsfonts,amssymb,amscd}
\usepackage{bbm}
\usepackage[enableskew]{youngtab}
\usepackage[mathscr]{euscript} 

\usepackage[bb=boondox]{mathalfa}

\usepackage{enumerate} 
\usepackage{tikz} 
\usetikzlibrary{calc,intersections,through, backgrounds,arrows,decorations.pathmorphing,backgrounds,positioning,fit,petri} 
\usetikzlibrary{calc}

\usepackage{tikz-cd} 
\usepackage[all]{xy}
\xyoption{knot}
\xyoption{poly}

\usepackage{hyperref}

\DeclareMathAlphabet{\mathpzc}{OT1}{pzc}{m}{it}

\usepackage[margin=1.16in]{geometry}

\newcommand{\sB}{\mathpzc{B}}
\newcommand{\SG}{\mathrm{S}}
\newcommand{\Ti}{\mathbf{Ti}}
\newcommand{\id}{\mathrm{id}}

\newcommand{\Spec}{\mathrm{Spec}}

\numberwithin{equation}{section}

\newtheoremstyle{notes} {} {} {} {} {\bfseries} {.} {.5em} {}

\theoremstyle{plain}

\newtheorem{prop}[subsubsection]{Proposition}

\newtheorem{lemma}[subsubsection]{Lemma}
\newtheorem{cor}[subsubsection]{Corollary}
\newtheorem{thm}[subsubsection]{Theorem}
\newtheorem{conj}[subsubsection]{Conjecture}

\theoremstyle{remark}

\newtheorem{rem}[subsubsection]{Remark} 
\newtheorem{ddef}[subsubsection]{Definition} 
\newtheorem{ex}[subsubsection]{Example} 

\newcommand{\mm}{\mathsf{m}}
\newcommand{\mn}{\mathsf{n}}
\newcommand{\mr}{\mathsf{r}}

\newcommand{\Ide}{\mathpzc{Id}}
\newcommand{\TL}{T\hspace{-0.6mm}L}

\newcommand{\TId}{\mathpzc{TId}}
\newcommand{\charr}{\mathrm{char}}

\newcommand{\cN}{{\mathcal{N}}}

\newcommand{\ev}{\mathrm{ev}}
\newcommand{\co}{\mathrm{co}}
\newcommand{\tto}{\twoheadrightarrow}
\newcommand{\mZ}{\mathbb{Z}}
\newcommand{\mQ}{\mathbb{Q}}
\newcommand{\mF}{\mathbb{F}}
\newcommand{\mN}{\mathbb{N}}
\newcommand{\mR}{\mathbb{R}}

\newcommand{\add}{\mathrm{add}}
\newcommand{\im}{\mathrm{im}}

\newcommand{\cF}{\mathcal{F}}

\newcommand{\op}{\mathrm{op}}

\newcommand{\Tr}{{\mathrm{Tr}}}
\newcommand{\sJ}{\mathpzc{J}}
\newcommand{\sI}{\mathpzc{I}}
\newcommand{\Par}{\mathsf{Par}}

\newcommand{\tr}{\mathsf{tr}}

\newcommand{\inde}{\mathrm{inde}}

\newcommand{\Sub}{\mathpzc{Sub}}

\newcommand{\OSp}{{\mathrm{OSp}}}
\newcommand{\GL}{{\mathrm{GL}}}
\newcommand{\SL}{{\mathrm{SL}}}
\newcommand{\Pe}{{\mathrm{Pe}}}

\newcommand{\RO}{{\underline{{\rm Rep}}} O_\delta}
\newcommand{\ROz}{{\underline{{\rm Rep}}_0} O_\delta}
\newcommand{\RG}{{\underline{{\rm Rep}}} GL_\delta}
\newcommand{\RGz}{{\underline{{\rm Rep}}_0} GL_\delta}
\newcommand{\RP}{{\underline{{\rm Rep}}} P}
\newcommand{\RPz}{{\underline{{\rm Rep}}_0} P}
\newcommand{\RS}{{\underline{{\rm Rep}}} S_t}

\newcommand{\RQ}{{\underline{{\rm Rep}}} Q}

\newcommand{\mG}{\mathbb{G}}

\usepackage{ stmaryrd }

\newcommand{\blambda}{\boldsymbol{\lambda}}
\newcommand{\bmu}{\boldsymbol{\mu}}
\newcommand{\bkappa}{\boldsymbol{\kappa}}
\newcommand{\bnu}{\boldsymbol{\nu}}
\newcommand{\cS}{{\mathbf{S}}}
\newcommand{\Tmod}{\mathscr{T}}
\newcommand{\mL}{{\mathbb{L}}}

\title[Tensor ideals]{Tensor ideals, Deligne categories and invariant theory}

\author{Kevin Coulembier}

\newcommand{\ind}{{\rm Ind}}
\newcommand{\res}{{\rm Res}}

\newcommand{\End}{{\rm End}}

\newcommand{\mk}{\Bbbk}
\newcommand{\mC}{\mathbb{C}}
\newcommand{\Hom}{{\rm Hom}}
\newcommand{\Span}{{\rm Span}}
\newcommand{\Stab}{{\rm Stab}}
\newcommand{\Rep}{{\rm Rep}}
\newcommand{\Nat}{{\rm Nat}}

\newcommand{\unit}{{\mathbbm{1}}}
\newcommand{\zero}{{\mathbb{0}}}
\newcommand{\cC}{{\mathbf{C}}}
\newcommand{\cT}{{\mathbf{T}}}
\newcommand{\cD}{{\mathbf{D}}}
\newcommand{\cJ}{{\mathcal{J}}}
\newcommand{\cM}{{\mathcal{M}}}
\newcommand{\Ob}{{\mathrm{Ob}}}

\newcommand{\Id}{{\mathrm{Id}}}
\newcommand{\Pmod}{\mathscr{P}}
\newcommand{\Mmod}{\mathscr{M}}
\newcommand{\Nmod}{\mathscr{N}}
\newcommand{\Wmod}{\mathscr{W}}
\newcommand{\Fmod}{\mathscr{F}}
\newcommand{\Gmod}{\mathscr{G}}

\newcommand{\oa}{\bar{0}}
\newcommand{\ob}{\bar{1}}

\newcommand{\sS}{\mathpzc{S}}
\newcommand{\sG}{\mathpzc{G}}

\newcommand{\Ab}{\mathbf{A\hspace{-0.4mm}b}}

\newcommand{\fg}{\mathfrak{g}}

\newcommand{\bO}{\mathbf{O}}
\newcommand{\bP}{\mathbf{P}}
\newcommand{\bT}{\mathbf{T}}

\keywords{Monoidal (super)category, tensor ideal, thick tensor ideal, Deligne category, algebraic (super)group, second fundamental theorem of invariant theory, tilting modules, quantum groups}

\subjclass[2010]{18D10, 17B45, 17B10, 15A72}

\begin{document} 
\date{} 
\begin{abstract}
We derive some tools for classifying tensor ideals in monoidal categories. We use these results to classify tensor ideals in Deligne's universal categories~$\RO$, $\RG$ and~$\RP$. These results are then used to obtain new insight into the second fundamental theorem of invariant theory for the algebraic supergroups of types~$A,B,C,D,P$. 

We also find new short proofs for the classification of tensor ideals in $\RS$ and in the category of tilting modules for~$\SL_2(\mk)$ with $\charr(\mk)>0$ and for~$U_q(\mathfrak{sl}_2)$ with $q$ a root of unity.
In general, for a simple Lie algebra $\fg$ of type ADE, we show that the lattice of such tensor ideals for $U_q(\fg)$ corresponds to the lattice of submodules in a parabolic Verma module for the corresponding affine Kac-Moody algebra.
	\end{abstract}

\maketitle 



\section*{Introduction}
Fix an algebraically closed field~$\mk$ of characteristic zero and $\delta\in\mk$. In~\cite{Deligne}, Deligne introduced monoidal categories~$\ROz$ and~$\RGz$ and their pseudo-abelian envelopes~$\RO$ and~$\RG$, and proved they are universal in the following sense. Let~$\cC$ be a symmetric $\mk$-linear monoidal category in which the endomorphism algebra of~$\unit_{\cC}$ is $\mk$ and take a self-dual object $X\in\Ob\cC$ of dimension $\delta\in\mk$. Then there exists a $\mk$-linear monoidal functor
$$\Fmod_X^{\cC}:\;\ROz\,\to\,\cC,$$
which maps the generator of~$\ROz$ to $X$. A similar property holds for~$\RGz$. This universality naturally leads to the question of what tensor ideals exist in~$\ROz$ and~$\RGz$, or equivalently in~$\RO$ and~$\RG$. These ideals correspond to the possible kernels of functors~$\Fmod_X^{\cC}$ above. Hence, a classification of tensor ideals yields a classification of the possible images of such $\Fmod_X^{\cC}$. More recently, analogous monoidal {\em super}categories~$\RP$ and~$\RQ$ were introduced in~\cite{ComesKuj, Kujawa, Vera}. 

The main result of the current paper is a classification of tensor ideals in~$\RO$, $\RG$ and~$\RP$. One way to phrase the conclusion of these classifications is that the ideals in the Deligne categories are precisely the kernels of the functors~$\Fmod_V^{\cC}$, where $\cC$ ranges over the categories of representations of algebraic affine supergroup schemes of types~$A$, for~$\RG$, types~$B,C,D$, for~$\RO$, and type $P$, for~$\RP$, and~$V$ is the natural representation.

As observed in~\cite[Section~1.2]{Comes}, a first step towards obtaining a classification of tensor ideals in the Deligne categories, is 
classifying the thick ideals in the split Grothendieck rings.
For~$\RG$, this was achieved by Comes in~\cite{Comes}, for~$\RO$ by Comes - Heidersdorf in~\cite{ComesHei} and for~$\RP$ by Coulembier - Ehrig in~\cite{PB3}. 
However, we will not rely on those results to obtain the classification of tensor ideals, meaning we also obtain a new proof for the classification of thick tensor ideals in the split Grothendieck rings.

Our classification of tensor ideals also leads to applications in the study of invariant theory for the supergroups $\OSp(m|2n)$, $\GL(m|n)$ and~$\Pe(n)$. For simplicity, we explain our results for one specific case, $\OSp(m|2n)$. By universality, there exists a monoidal functor
$$\Fmod_{m,n}:\;\ROz\,\to\,\Rep_{\mk}\OSp(m|2n),\quad\mbox{if $\delta=m-2n$,}$$
where the objects in the image are precisely the tensor powers $V^{\otimes i}$ of the natural representation $V=\mk^{m|2n}$.
As proved by Lehrer - Zhang, the functor~$\Fmod_{m,n}$ is full, which is one of the incarnations of the {\em first} fundamental theorem of invariant theory for~$\OSp(m|2n)$, see \cite{ES1, LZ-FFT, Vera, Sergeev}, in the terminology of H. Weyl. For the other supergroups, the first fundamental theorem is given in~\cite{Berele, DLZ, Kujawa, Moon}. Descriptions of the kernel of the full functor~$\Fmod_{m,n}$ are usually referred to as the {\em second} fundamental theorem. Our classification of ideals in~$\ROz$ naturally yields a description of those kernels. We use this description to re-obtain and extend some results on the second fundamental theorem for~$\OSp(m|2n)$ and~$\GL(m|n)$ and to obtain for the first time the second fundamental theorem for~$\Pe(n)$.

We pay specific attention to the description of the surjective algebra morphisms
$$\phi^r\,:\;B_r(\delta)\;\tto\; \End_{\OSp(V)}(V^{\otimes r}),\qquad\mbox{for~$r\in\mN$,}$$
induced from $\Fmod_{m,n}$, with~$B_r(\delta)$ the Brauer algebra. We establish when $\phi^r$ is an isomorphism, as recently obtained in~\cite{Yang} through different methods.
Furthermore, we prove that the kernel is always generated by a single element, as a two-sided ideal. For~${\rm Sp}(2n)$ this was proved by Hu - Xiao in~\cite{HX}, for~${\rm O}(m)$ by Lehrer - Zhang in~\cite{LZ1}, and for~$\OSp(1|2n)$ by Zhang in~\cite{Yang}. In all other cases, this is new. We also prove that this generating element can be chosen as an idempotent if and only if $m\le 1$ or~$n=0$. That this is possible for~${\rm Sp}(2n)$ and~${\rm O}(m)$ was proved in~\cite{HX, LZ1}, but is new for~$\OSp(1|2n)$. Again, our analogous results for~$\Pe(n)$ seem to be entirely new.

The paper is organised as follows. In Sections \ref{SecPrel1} and \ref{SecPrel2} we recall some facts on (monoidal) categories. 
Section~\ref{SecPrincipal}
is concerned with some observations made in~\cite{AK}. We reformulate these into the statement that, under certain rigidity conditions on a monoidal category~$\cC$, the lattice of tensor ideals is isomorphic to the lattice of subfunctors of $\cC(\unit,-)$, which is essentially a lattice of submodules of a module over a ring. The most striking manifestation of this occurs for the category of tilting modules for a quantum group, for which results of \cite{KL, Soergel} then show that the lattice of tensor ideals is isomorphic to the lattice of submodules in a parabolic Verma module for the corresponding affine Kac-Moody algebra.
 We also demonstrate the usefulness of the general statement by applying it to reduce the classification of tensor ideals in the Temperley-Lieb category, resp. Deligne's category~$\RS$, to the study of one particular cell module over the Temperley-Lieb algebra, resp. partition algebra. In this way, we prove in a very elementary way that, in both cases, the only proper tensor ideal is the ideal of negligible morphisms, as first proved by Goodman - Wenzl in~\cite{GW}, resp. Comes - Ostrik in~\cite{ComesOst}. 

In Section~\ref{SecFibres}, we investigate necessary and sufficient conditions under which the tensor ideals in a Krull-Schmidt monoidal category~$\cC$ are in natural bijection with the thick ideals in its split Grothendieck ring $[\cC]_{\oplus}$. We apply these results to describe a setup where the classification of tensor ideals becomes a combinatorial exercise.

In Section~\ref{SecModular}, we apply the results of Section~\ref{SecFibres} to the monoidal category of tilting modules over a reductive group in positive characteristic. We prove that this category will generally contain infinitely many tensor ideals and these will {\em not} be in bijection with the thick ideals in the Grothendieck ring. We show that $\SL_2$ provides an exception to the latter behaviour, by classifying all tensor ideals. One way of formulating our result is that the tensor ideals are precisely the kernels of the canonical functors
$$\Ti(G)\to \Stab G_rT,\quad\mbox{with $r\in\mZ_{>0}$,}$$
with $\Ti(G)=\Ti(\SL_2)$ the category of $\SL_2$-tilting modules, $\Stab G_rT$ the stable module category of $\Rep_{\mk}G_rT$ and all other notation as in \cite{Jantzen}.

In Section~\ref{SecBoring} we interpret some decomposition multiplicities of cell modules for Brauer type algebras from~\cite{PB2, BrMult, Martin} in 
terms of Deligne categories. This is precisely the input that will be needed in Sections~\ref{SecClass} and~\ref{SecSFT}, to use the general results of Section~\ref{SecFibres} to obtain our main results on the classification of tensor ideals and the second fundamental theorem.

In Section~\ref{SecFurther} we discuss some further applications of our general results on tensor ideals. We give new proofs for some results on the second fundamental theorem of invariant theory for the symmetric group, by Jones in~\cite{Jones2} and Benkart - Halverson in~\cite{BH}. We also obtain the classification of tensor ideals in the quantum analogue ${\underline\Rep} U_q(\mathfrak{gl}_\delta)$ of $\RG$, for generic~$q$, using recent results in~\cite{Brundan}. We also state a conjecture about $\RQ$ and use the conclusions in Section~\ref{SecModular} to point out some expected difficulties concerning modular analogues of our results on Deligne categories. In Appendix~\ref{SecSuper}, we briefly discuss the extension of our general results to monoidal {\em super}categories.

\part{General considerations}

\section{Rings, partitions and categories}\label{SecPrel1}
We set~$\mN=\{0,1,2,\ldots\}$ and denote by $\Ab$ the category of abelian groups.

\subsection{Rings}\label{SecPrel11} We will not require rings to be unital.

\subsubsection{}\label{DefThick} Let~$R$ be a ring which is free as an abelian group, with basis $\sG$. A left ideal $I$ in~$R$ is {\bf thick}
with respect to~$\sG$ if, as an abelian group, it is (freely) generated by a subset of~$\sG$.
We denote by~$\Ide(R;\sG)$ the set of thick left ideals, with partial order describing inclusion. 
\subsubsection{}\label{trace}
For a ring $R$, a left $R$-module~$M$ is an abelian group with surjective map $R\times M\to M$ satisfying the three ordinary properties.
We denote by~$R$-Mod the category of left $R$-modules.
For~$M\in R$-Mod, we denote by~$\Sub(M)$ the set of submodules, partially ordered with respect to inclusion.
For~$M,N\in R$-Mod, the {\bf trace of~$M$ in~$N$} is the submodule of~$N$
$$\Tr_{M}N\;:=\;\sum_{f:M\to N}\im(f).$$

\subsection{Partitions}
We denote the set of all partitions by~$\Par$. The empty partition is~$\varnothing$. The transpose (conjugate) of a partition $\lambda$ is denoted by $\lambda^t$.

\subsubsection{}\label{DualPart} 
For~$a,b\in\mN$, we say that~$\lambda,\mu\in\Par$ are $a\times b$-{\bf dual} if
$$\lambda_i+\mu_{a+1-i}\;=\; b,\quad\mbox{for~$1\le i\le a$,}\quad\mbox{and}\quad \lambda_{a+1}=0=\mu_{a+1}.$$
Each partition $\lambda\subset (b^a)$ has a unique $a\times b$-dual. 

\subsubsection{}\label{SecSG} Denote by~$\SG_n$ the symmetric group on $n$ symbols, for~$n\in\mN$. For~$\mk$ an algebraically closed field of characteristic zero, the simple modules over~$\mk\SG_n$ are labelled by the partitions~$\lambda\vdash n$. It will be convenient to denote the (simple) Specht module corresponding to $\lambda\vdash n$ simply by $\lambda$.

We consider the embedding $\SG_r\times \SG_s<\SG_{r+s}$. For~$\lambda,\mu,\nu\in\Par$ with~$|\nu|=|\lambda|+|\mu|$, the Littlewood-Richardson coefficient $c^{\nu}_{\lambda\mu}$ is defined through the relation
$$\ind^{\mk \SG_{r+s}}_{\mk \SG_r\otimes \mk\SG_s}( \lambda\boxtimes \mu)\;\simeq\; \bigoplus_{\nu\vdash r+s } \nu^{\oplus c_{\lambda\mu}^{\nu}}.$$
\begin{lemma}
\label{LemRect}Fix $a,b\in\mN$ and~$\lambda,\mu\in\Par$. 
\begin{enumerate}[(i)]
\item If~$c_{\lambda\mu}^{\nu}\not=0$ for some~$\nu\supset (b^a)$, then the following equivalent properties hold:
\begin{itemize}
\item $\lambda_i+\mu_{a+1-i}\;\ge\; b,\quad\mbox{for~$1\le i\le a$;}$
\item $\lambda_j^t+\mu_{b+1-i}^t\;\ge\; a,\quad\mbox{for~$1\le j\le b$.}$
\end{itemize}
\item If~$|\lambda|+|\mu|=ab$, then
$$c_{\lambda\mu}^{(b^a)}\;=\;\begin{cases}1&\mbox{if $\lambda$ and~$\mu$ are $a\times b$-dual}\\
0&\mbox{otherwise}.
\end{cases}$$
\end{enumerate}
\end{lemma}
\begin{proof}
These are direct applications of the Littlewood-Richardson rule.
\end{proof}

\subsection{Categories}

With the exception of $\Ab$ and functor categories, we {\em only consider categories~$\cC$ which are equivalent to small categories}. For the entire subsection, let $\cC$ be a preadditive (enriched over~$\Ab$) category. We denote by~$\inde\cC$ the {\em set} of isomorphism classes of indecomposable objects in~$\cC$.

\subsubsection{}\label{IdPA}
An {\bf ideal} $\cJ$ in~$\cC$ consists of subgroups $\cJ(X,Y)$ of~$\cC(X,Y)$, for all~$X,Y\in\Ob\cC$, such that for all~$X,Y,Z,W\in\Ob\cC$, and~$g\in\cC(X,Y)$ and~$h\in\cC(Z,W)$, we have that
$$f\in\cJ(Y,Z)\quad\mbox{implies}\quad f\circ g\in \cJ(X,Y)\quad\mbox{and}\quad h\circ f \in \cJ(Y,W). $$
The typical example of an ideal is the {\bf kernel} $\ker \Fmod$ of an additive functor~$\Fmod:\cC\to\cD$.


For an ideal $\cJ$, we have the quotient category~$\cC/\cJ$ which has as objects~$\Ob\cC$, but as morphism groups~$\cC(X,Y)/\cJ(X,Y)$, for all~$X,Y\in \Ob\cC$.

\subsubsection{} 

The category~$\cC$ is {\bf karoubian} if for every~$X\in\Ob\cC$ and idempotent $e\in\cC(X,X)$, there exists $Y\in\Ob\cC$, with~$f\in\cC(X,Y)$ and~$g\in\cC(Y,X)$ such that~$f\circ g=1_Y$ and~$g\circ f=e$. 

If~$\cC$ is additive we denote the {\bf zero object} (the empty biproduct) by $\zero$, in order to avoid confusion with other occurrences of the symbol $0$. 

An additive category~$\cC$ is {\bf Krull-Schmidt} if 
$\cC(X,X)$ is a local ring for every $X\in\inde\cC$ and every object is a finite direct sum of indecomposable objects. Then $\cC$ is karoubian and every object has a unique (up to isomorphism) decomposition into a finite direct sum (biproduct) of indecomposable objects.
For a Krull-Schmidt category~$\cC$ and~$X\in\Ob\cC$, we write~$\add(X)$ for the class of objects in~$\cC$ which are direct sums of direct summands of~$X$. We also write $X\inplus Y$ to denote that $X$ is a direct summand of $Y$.

\subsubsection{Modules over~$\cC$}\label{CMod}
 The category~$\cC$-Mod is the category of additive functors
$\Mmod: \cC\to\Ab.$
The morphism groups~$\Nat(\Mmod,\Nmod)$ consist of all natural transformations~$\Mmod\Rightarrow \Nmod$. 
An additive functor~$\Fmod:\cC\to\cD$ induces an additive functor 
$$\cD\mbox{-Mod}\,\to\,\cC\mbox{-Mod};\quad \Mmod\mapsto \Mmod\circ \Fmod,$$
which is an equivalence when $\Fmod$ is an equivalence.

Now assume that~$\cC$ is small, or replace it by an equivalent small category. The group
$$\mZ[\cC]\;=\; \bigoplus_{X,Y\in\Ob\cC}\cC(X,Y)$$
is a ring with multiplication for~$f\in \cC(X,Y)$ and~$g\in \cC(Z,W)$ given by
$fg=f\circ g$, if $W=X$, and $fg=0$ otherwise.
There is an equivalence between $\cC$-Mod and~$\mZ[\cC]$-Mod,
$$\cC\mbox{-Mod}\;\stackrel{\sim}{\to}\; \mZ[\cC]\mbox{-Mod},\quad \Mmod\mapsto \bigoplus_{X\in \Ob\cC}\Mmod(X).$$
This shows in particular that~$\cC$-Mod is abelian.  
We will also use the notation for modules over rings from Section~\ref{SecPrel11} in this context. For instance, for an additive functor~$\Mmod:\cC\to\Ab$, we denote by $\Sub_{\cC}(\Mmod)$ the set of subfunctors.  
\subsubsection{}\label{PmodZ}  For~$Z\in \Ob\cC$, we have
$$\Pmod^{\cC}_Z=\Pmod_Z\;:=\;\cC(Z,-)\;\in \cC\mbox{-Mod.}$$
This module is projective in~$\cC$-Mod, which follows for instance from the Yoneda Lemma for preadditive categories. 
The Yoneda lemma also implies that for~$Z,W,Y\in\Ob\cC$ we have
\begin{equation}\label{eqTr}
\Tr_{W}\Pmod_Z(Y):=\Tr_{\Pmod_W}\Pmod_Z(Y)\;=\;\Span_{\mZ}\{a\circ b\;|\;  a\in\cC(W,Y),\;\,b\in\cC(Z,W)\}.\end{equation}
For~$\sS\subset\Ob\cC$, we set
$$\Tr_{\sS}\Pmod_{Z}\;:=\;\sum_{X\in \sS}\Tr_{X} \Pmod_{Z}\;\in\;\cC\mbox{{\rm-Mod}}.$$
We refer to the above submodules, for arbitrary $\sS\subset\Ob\cC$, as the {\bf trace submodules} of~$\Pmod_Z$. 

\begin{lemma}\label{Lem2Tr}
Assume~$\cC$ is Krull-Schmidt. For~$Z,W\in\inde\cC$ and a non-zero $f\in\cC(Z,W)$, let $\Mmod$ be the submodule of $\Pmod_Z$ generated by $f$.
Then $\Mmod\not=\Tr_{\sS}\Pmod_{Z}$ for any $\sS\subset\inde\cC\backslash\{W\}$.
\end{lemma}
\begin{proof}It follows easily that $\Mmod=\Tr_{\sS}\Pmod_{Z}$ implies that $f=\alpha\circ f$, for some $\alpha\in\cC(W,W)$ which factors through a direct sum of objects in $\sS$. Since $\cC(W,W)$ is local, either $\alpha$ or $1_W-\alpha$ is an isomorphism. If $\alpha$ is an isomorphism, then $W$ is a direct summand of a direct sum of objects in $\sS$, a contradiction. So $1_W-\alpha$ is an isomorphism, but this contradicts $0=(1_W-\alpha)\circ f$.
\end{proof}

\subsubsection{Split Grothendieck group}For~$\cC$ additive, the {\bf split Grothendieck group}~$[\cC]_{\oplus}\in\Ob\Ab$ of~$\cC$ is the abelian group with generators the isomorphism classes~$[X] $ of objects~$X$ in~$\cC$, and relations~$[X]=[ Y]+[ Z]$, whenever $X\simeq Y\oplus Z$. If~$\cC$ is Krull-Schmidt, then $[\cC]_{\oplus}$ is isomorphic to the free abelian group with basis $\inde\cC$ and we have $X\simeq Y$ if and only if $[X]=[Y]$.


\section{Monoidal categories}\label{SecPrel2}


\subsection{Basic definitions}

\subsubsection{}\label{DefMC}A {\bf strict monoidal preadditive category} is a triple~$(\cC,\otimes,\unit)$ comprising
 a preadditive category~$\cC$ with bi-additive functor~$\otimes: \cC\times\cC\to \cC$
and object $\unit\in\Ob\cC$, such that
\begin{itemize}
\item we have equalities of functors~$ \unit\otimes-=\Id$, resp. $-\otimes\unit= \Id$;
\item we have an equality of functors
$$(-\otimes -)\otimes -\;=\; -\otimes (-\otimes -):\quad \cC\times\cC\times\cC\to \cC.$$
\end{itemize}
From now on, we will {\em leave out the reference to preadditivity} when speaking about monoidal categories.
It follows immediately from the definition that~$K:=\cC(\unit,\unit)$ is a commutative ring and $\cC$ is a $K$-linear category, with $\lambda f:=\lambda\otimes f$, for $\lambda\in K$ and $f$ an arbitrary morphism. 

On a monoidal category~$\cC$, a {\bf braiding} $\gamma$ is a bi-natural family of isomorphisms~$\gamma_{XY}:X\otimes Y\stackrel{\sim}{\to}Y\otimes X$ which satisfy the hexagon identities. We automatically have $\gamma_{\unit,X}=1_X=\gamma_{X,\unit}$. If $\cC$ admits a braiding it follows that $\otimes$ is bilinear with respect to $K=\cC(\unit,\unit)$.

We call~$\Pmod_{\unit}$, as in~\ref{PmodZ}, the {\bf principal $\cC$-module.} 

\subsubsection{} When we do not require the monoidal category to be strict, the equalities of functors are replaced by isomorphisms, known as associators and unitors, satisfying the commuting triangle and pentagon diagram condition. Since each monoidal category is equivalent to a strict one, by Mac Lane's Coherence Theorem, we will often pretend a monoidal category is strict.

\subsubsection{} If~$\cC$ is an additive monoidal category, the split Grothendieck group $[\cC]_{\oplus}$ naturally becomes a ring with respect to the multiplication~$[\cC]_{\oplus}\otimes_{\mZ}[\cC]_{\oplus}\to [\cC]_{\oplus}$ induced from the bi-additive functor~$\cC\times\cC\to\cC$. This is the {\bf split Grothendieck ring} of~$\cC$.

\subsection{Tensor ideals}
Let~$\cC$ be a monoidal category.
\begin{ddef}
A {\bf left-tensor ideal $\cJ$} in~$\cC$ is an ideal $\cJ$ such that, for all~$X,Y,Z\in\Ob\cC$,
$$f\in\cJ(X,Y)\quad\mbox{implies that}\quad 1_Z\otimes f\in \cJ(Z\otimes X,Z\otimes Y).$$
\end{ddef}
For~$\cJ$ as in the definition, it follows that $g\otimes f$ belongs to~$\cJ$, for an arbitrary morphism $g$, if $f$ belongs to~$\cJ$. 
When $\cJ$ is a left-tensor ideal in a braided monoidal category~$\cC$, or more generally, when $\cJ$ is two-sided, the quotient category~$\cC/\cJ$ is naturally monoidal.

\subsubsection{}Since we will only consider left-tensor ideals (except in braided monoidal categories, when left-tensor ideals are automatically two-sided), we will often refer to left-tensor ideals in monoidal categories simply as tensor ideals.
We denote the set of left-tensor ideals in~$\cC$ by~$\TId(\cC)$. This set is partially ordered with respect to the obvious notion of inclusion.

\subsubsection{} Next, we will introduce thick tensor ideals for Krull-Schmidt monoidal categories. These can either be described as certain strictly full subcategories in~$\cC$, subsets of~$\Ob\cC$, or ideals in the split Grothendieck ring~$[\cC]_{\oplus}$. The first one is perhaps most common in the literature, but for our purposes the latter two are most convenient.

\begin{ddef}\label{DefOI}
A {\bf thick left-tensor~$\Ob$-ideal} in a Krull-Schmidt monoidal category~$\cC$ is an isomorphism closed subset~$\sI$ of~$\Ob\cC$ such that
\begin{enumerate}[(a)]
\item $X\oplus Y\in \sI$ if and only if $X,Y\in \sI$;
\item $X\in \sI$ implies~$Y\otimes X\in\sI$, for all~$Y\in \Ob \cC$.
\end{enumerate}
\end{ddef}

\subsubsection{}\label{ThickIdeals2}We have an obvious identification between thick left-tensor~$\Ob$-ideals in~$\cC$ and thick left ideals, as in~\ref{DefThick}, in the ring $[\cC]_{\oplus}$ with respect to the basis $\{[X]\,|\,X\in\inde\cC\}$. Therefore, we denote the partially ordered set of such ideals by
$$\Ide([\cC]_\oplus)\;:=\;\Ide([\cC]_{\oplus},\inde\cC).$$
Specific kinds of thick tensor~$\Ob$-ideals lead to `tensor triangulated geometry', see e.g.~\cite{Balmer}.

\subsection{Rigidity} Fix a strict monoidal category~$\cC$.

\subsubsection{}A {\bf right dual} of~$X\in\Ob\cC$ is a triple~$(X^\vee,\ev_X,\co_X)$ with~$X^\vee\in \Ob\cC$ and morphisms
$$\ev_X:X\otimes X^\vee\to\unit\qquad\mbox{and}\qquad \co_X:\unit\to X^\vee\otimes X,$$
known as the evaluation and coevaluation, which satisfy
\begin{equation}\label{eqdual}(\ev_X\otimes 1_X)\circ (1_X\otimes \co_X)=1_X\qquad\mbox{and}\qquad (1_{X^{\vee}}\otimes\ev_X)\circ(\co_X\otimes 1_{X^\vee})=1_{X^\vee}. \end{equation}
A right dual is unique, up to isomorphism. A strict monoidal category~$\cC$ is {\bf right rigid} if each object admits a right dual.

\subsubsection{}\label{secdual} We recall some well-known results from \cite[Section~6.1]{AK}. Assume $X\in\Ob\cC$ admits a right dual. 
We have a canonical group isomorphism
\begin{equation}\label{eqiota}\iota_{XY}:\;\cC(\unit,X^\vee\otimes Y)\,\stackrel{\sim}{\to}\,\cC(X,Y),\qquad\phi\mapsto (\ev_X\otimes 1_Y)\circ (1_X\otimes \phi),\end{equation}
with inverse given by
\begin{equation}\label{eqiota1}\iota^{-1}_{XY}(f)\;=\; (1_{X^\vee}\otimes f)\circ\co_X,\qquad\mbox{for~$f\in \cC(X,Y)$.}\end{equation}
Equation~\eqref{eqiota1} shows in particular that the left $\cC(X^\vee\otimes X,X^\vee\otimes X)$-module~$\cC(\unit, X^\vee\otimes X)$ is generated by $\co_X$. Assume~$X,Y\in\Ob\cC$ admit right duals. Then $Y\otimes X$ admits a right dual~$X^\vee\otimes Y^\vee$ with
\begin{equation}\label{eqXYd}\ev_{Y\otimes X}\;=\;\ev_Y\circ (1_Y\otimes\ev_X\otimes 1_{Y^\vee})\qquad\mbox{and}\qquad \co_{Y\otimes X}\;=\;(1_{X^\vee}\otimes \co_Y\otimes 1_X)\circ\co_{X}.\end{equation}

\subsection{Assumptions} For future reference, we list some assumptions on (preadditive) monoidal categories~$\cC$ with $K:=\cC(\unit,\unit)$ we will frequently use.
\begin{enumerate}[(I)]
\item $\cC$ is right rigid.
\item $\cC$ is Krull-Schmidt.
\item $K$ is a field.
\item $\cC(X,Y)$ is a finitely generated $K$-module, for all $X,Y\in\Ob\cC$.
\item $\cC$ admits a braiding.
\end{enumerate}
Note that by convention we assume that Krull-Schmidt categories are additive and karoubian. If~$\cC$ is additive, karoubian and (III), (IV) are satisfied, (II) is automatically satisfied.

\subsection{Negligible morphisms}
\subsubsection{}\label{Condtr}Assume that~$\cC$ satisfies (I) and (V), with~$K:=\cC(\unit,\unit)$ and consider a braiding $\gamma$. We have a morphism of~$K$-modules
$$\tr:\;\cC(X,X)\,\to\, K,\qquad\mbox{for all~$X\in\Ob\cC$},$$
where the {\bf trace} $\tr(f)$ of any morphism $f\in \cC(X,X)$ is the composition 
$$\unit\;\stackrel{\co_X}{\to}\;X^\vee\otimes X\;\stackrel{1_{X^\vee}\otimes f}{\to}\;X^\vee\otimes X\;\stackrel{\gamma_{X^\vee X}}{\to}\;X\otimes X^\vee\;\stackrel{\ev_X}{\to}\;\unit.$$
As stated in \cite[(7.2)]{AK} in slightly less generality, we have $\tr(f\circ g)=\tr(g\circ f)$, for $f\in\cC(X,Y)$ and $g\in\cC(Y,X)$. Indeed, one calculates using \eqref{eqdual} and naturality of $\gamma$ that
\begin{eqnarray*}
\tr(g\circ f)&=&\ev_X\circ\gamma_{X^\vee X}\circ (((1_{X^\vee}\otimes \ev_Y)\circ(1_{X^\vee}\otimes f\otimes 1_{Y^\vee})\circ (\co_X\otimes 1_{Y^\vee})) \otimes g)\circ \co_Y\\
&=&\ev_X\circ (g\otimes((1_{X^\vee}\otimes \ev_Y )\circ(1_{X^\vee}\otimes f\otimes 1_{Y^\vee})\circ (\co_X\otimes 1_{Y^\vee})) )\circ \gamma_{Y^\vee Y}\circ\co_Y\\
&=&\ev_Y\circ(f\circ g\otimes 1_{Y^\vee}) \circ \gamma_{Y^\vee Y}\circ\co_Y\;=\;\tr(f\circ g).
\end{eqnarray*}

\subsubsection{}A morphism $f\in\cC(X,Y)$ is {\bf negligible} if $\tr(g\circ f)=0$, for all~$g\in \cC(Y,X)$. 
The following lemma is stated in \cite[Proposition~7.1.4]{AK} in slightly less generality. For completeness, a proof will be given in Section~\ref{SecProofs}. The lemma implies that the notion of negligibility does not depend on the choice of braiding.
\begin{lemma}\label{PropAK} Let~$\cC$ be a monoidal category satisfying (I), (III) and~(V).
The unique maximal tensor ideal in~$\cC$ consists of all negligible morphisms.
\end{lemma}

\subsection{Completions}

\subsubsection{}Let $\cC$ be a preadditive category. We denote its additive envelope, see e.g. \cite[Section~2.5]{ComesWil} or \cite[Section~1.2]{AK}, by~$\cC^{\oplus}$. We have a canonical equivalence between $\cC$-Mod and~$\cC^{\oplus}$-Mod. We denote the Karoubi envelope (idempotent completion), see e.g. \cite[Section~2.6]{ComesWil} or \cite[Section~1.2]{AK}, of~$\cC$ by~$\cC^{\sharp}$. 
We have a canonical equivalence between $\cC$-Mod and~$\cC^{\sharp}$-Mod. Objects in~$\cC^{\sharp}$ will be denoted by~$(X,e)$, with~$X\in\Ob\cC$ and~$e$ an idempotent in~$\cC(X,X)$.

\subsubsection{}For a preadditive category~$\cC$, we call~$\cC^{\oplus\sharp}\simeq\cC^{\sharp\oplus}$ the {\bf pseudo-abelian envelope} of~$\cC$. If~$\cC$ is monoidal, then $\cC^{\oplus\sharp}$ is also a monoidal category, which extends the monoidal structure of the subcategory~$\cC$. Restriction along the fully faithful embedding $\cC\to\cC^{\oplus\sharp}$ yields an isomorphism $\TId(\cC^{\oplus\sharp})\stackrel{\sim}{\to}\Id(\cC)$, see e.g. \cite[Lemme~1.3.10]{AK}.


\section{Tensor ideals as submodules of the principal module}\label{SecPrincipal}
In this section, we bring some ideas from \cite[Section~6]{AK} into the form that we will require.

\subsection{An isomorphism of lattices} Consider a preadditive category~$\cC$. For any ideal $\cJ$ and $Z\in \Ob\cC$, we can interpret~$\cJ(Z,-)$ as a functor~$\cC\to\Ab$, where for~$f\in \cC(X,Y)$, we set
$$\cJ(Z,f):\;\cJ(Z, X)\to\cJ(Z, Y),\;\; \alpha\mapsto f\circ\alpha.$$
By construction, $\cJ(Z,-)$ is a submodule of~$\Pmod_{Z}\in \cC\mbox{-Mod}$. We apply this to a monoidal category~$\cC$, for~$Z=\unit\in\Ob\cC$ and restrict $\cJ$ to left-tensor ideals.

\begin{thm}\label{Thm1}
For a right rigid monoidal category~$\cC$, the assignment
$$\Psi:\TId(\cC)\to \Sub_{\cC}(\Pmod_{\unit}),\quad \cJ\mapsto\cJ(\unit,-),$$
yields an isomorphism of partially ordered sets.
\end{thm}

\begin{rem}
An alternative formulation of this theorem can be found in~\cite[Section~6.3]{AK}.
\end{rem}
Before proceeding to the proof of Theorem~\ref{Thm1}, we formulate some consequences and a remark.
\begin{cor}\label{Cor4Bru}
Let~$\cC$ and~$\cD$ be right rigid monoidal categories. An equivalence~$\Fmod:\cC\to\cD$ of preadditive categories with~$\cF(\unit_{\cC})= \unit_{\cD}$ induces an isomorphism of partially ordered sets $\cF:\TId(\cD)\stackrel{\sim}{\to}\TId(\cC)$, where
$$\cF(\cJ)(X,Y)=\iota_{XY}\left(\Fmod^{-1}( \cJ(\unit_\cD,\Fmod(X^\vee\otimes Y)))\right),\qquad\mbox{for all~$\cJ\in\TId(\cD)$ and~$X,Y\in \Ob\cC$}.$$ 
\end{cor}
\begin{proof}
We consider the equivalence from $\cD$-Mod to~$\cC$-Mod as in~\ref{CMod}, which induces an isomorphism from $\Sub(\Pmod_{\unit}^{\cD})$ to~$\Sub(\Pmod_{\unit}^{\cC})$. The statement thus follows from the isomorphism in Theorem~\ref{Thm1} and the expression for its inverse given in Proposition~\ref{PropPhi} below.
\end{proof}

\begin{cor}
If $\cC$ satisfies (I), (II) and (III), the following are equivalent.
\begin{enumerate}[(i)]
\item $\cC$ has precisely one proper tensor ideal.
\item There is $\unit\not=X\in\inde\cC$, with $\cC(\unit,X)$ a (non-zero) simple $\cC(X,X)$-module
and $\cC(\unit,Y)=0$ for all $Y\in\inde\cC\backslash\{\unit,X\}$.
\end{enumerate}
\end{cor}
\begin{proof}
We freely use Theorem~\ref{Thm1}. Assume that $\Pmod_{\unit}$ has precisely one proper submodule. In particular, we must have some $\unit\not=X\in\inde\cC$ with non-zero $f\in\cC(\unit,X)$. Denote by $\Mmod$ the submodule of $\Pmod_{\unit}$ generated by $f$. By Lemma~\ref{Lem2Tr}, $\Mmod$ is different from $\Tr_Y\Pmod_\unit$, for all $Y\in\inde\cC\backslash\{X\}$. Hence we find $\Tr_Y\Pmod_{\unit}=0$ for all $Y\in\inde\cC\backslash\{\unit,X\}$, implying $\cC(\unit,Y)=0$.
If $\cC(\unit,X)$ would not be a simple $\cC(X,X)$-module we could easily construct proper submodules of $\Tr_X\Pmod_\unit$. Hence (i) implies (ii).

If (ii) is satisfied, it follows easily that $\Pmod_{\unit}$ has precisely one proper submodule, namely $\cM$ determined by $\cM(\unit)=0$ and $\cM(X)\not=0$.
 \end{proof}

\begin{rem}
If $\cC$ is any monoidal category we can define a map $\Phi:\Sub_{\cC}(\Pmod_{\unit})\to\TId(\cC)$, by assigning to a submodule $\Mmod$ of $\Pmod_{\unit}$ the tensor ideal generated by the morphisms in $\Mmod$. It follows easily that $\Psi\circ\Phi$ is the identity on $\Sub_{\cC}(\Pmod_{\unit})$. Consequently, $\Psi$ is always surjective and $\Phi$ injective. However, they need not be isomorphisms when $\cC$ is not rigid.
\end{rem}

\subsection{Proofs}\label{SecProofs}

By definition, for a submodule~$\Mmod$ of~$\Pmod^{\cC}_{\unit}$, we have $\Mmod(X)\subset \cC(\unit,X)$, for all~$X\in\Ob\cC$.
\begin{prop}\label{PropPhi}${}$
\begin{enumerate}[(i)]
\item 
For a submodule~$\Mmod$ of~$\Pmod_{\unit}$, we define $\cJ_{\Mmod}$ as
$$\cJ_{\Mmod}(X,Y):=\iota_{XY}\left({\Mmod}(X^\vee\otimes Y)\right),\qquad\mbox{for all~$X,Y\in\Ob\cC$}.$$
Then $\cJ_{\Mmod}$ is a left-tensor ideal in~$\cC$. 

\item Define a map $\Phi:\Sub_{\cC}(\Pmod_{\unit})\to\TId(\cC)$, by~${\Mmod}\mapsto \cJ_{\Mmod}$, then $\Phi\circ\Psi$ and~$\Psi\circ\Phi$ are the identity.
\end{enumerate}
\end{prop}

We start the proof of this proposition with three lemmata.

\begin{lemma}\label{Lemfg}
Consider $\phi\in\cC(\unit,X^\vee\otimes Y)$ for~$X,Y\in\Ob\cC$.
\begin{enumerate}[(i)]
\item For~$Z\in \Ob\cC$ and~$g\in\cC(Y,Z)$, we have
$$g\circ\iota_{XY}(\phi)\;=\;\iota_{XZ}((1_{X^\vee}\otimes g)\circ\phi).$$
\item For~$W\in\Ob\cC$ and~$f\in\cC(W,X)$, we have
$$\iota_{XY}(\phi)\circ f\;=\;\iota_{WY}((\chi\otimes 1_Y)\circ \phi),\quad\mbox{with }\;\chi:=(1_{W^\vee}\otimes\ev_X)\circ (\iota^{-1}_{WX}(f)\otimes 1_{X^\vee}).$$
\end{enumerate}
\end{lemma}
\begin{proof}
Part~(i) follows from a direct application of equation~\eqref{eqiota}.

For part~(ii), we set~$\overline{f}:=\iota^{-1}_{WX}(f)$. Then we have
\begin{eqnarray*}
\iota_{XY}(\phi)\circ f&=&(\ev_X\otimes 1_Y)\circ (1_X\otimes \phi)\circ (\ev_W\otimes 1_X)\circ (1_W\otimes \overline{f})\\
&=& (\ev_W\otimes \ev_X\otimes 1_Y)\circ (1_W\otimes \overline{f}\otimes \phi)\\
&=&\iota_{WY}\left((1_{W^\vee}\otimes \ev_X\otimes 1_Y)\circ ( \overline{f}\otimes \phi)\right),
\end{eqnarray*}
from which the claim follows easily.
\end{proof}

\begin{lemma}\label{LemTenZ}
Consider $\phi\in\cC(\unit,X^\vee\otimes Y)$ for~$X,Y\in\Ob\cC$. For~$Z\in\Ob\cC$, we have
$$1_Z\otimes \iota_{XY}(\phi)=\iota_{Z\otimes X,Z\otimes Y}(\psi),\quad\mbox{with~$\psi=(1_{X^\vee}\otimes \co_Z\otimes 1_Y)\circ \phi.$}$$
\end{lemma}
\begin{proof}
Equations~\eqref{eqdual} and~\eqref{eqXYd} imply
\begin{eqnarray*}
1_Z\otimes \iota_{XY}(\phi)&=& (\ev_Z\otimes 1_Z\otimes 1_Y)\circ (1_Z\otimes\ev_X\otimes\co_Z\otimes 1_Y)\circ (1_Z\otimes 1_X\otimes \phi)\\
&=&( \ev_{Z\otimes X}\otimes 1_Z\otimes 1_Y)\circ (1_Z\otimes 1_X\otimes 1_{X^\vee}\otimes \co_Z\otimes 1_Y)\circ (1_Z\otimes 1_X\otimes \phi),
\end{eqnarray*}
which proves the claim.
\end{proof}

\begin{lemma}\label{LemJJ}
For~$\cJ$ a tensor ideal in~$\cC$, we have
$$\cJ(X,Y)\;=\;\iota_{XY}(\cJ(\unit,X^\vee\otimes Y)).$$
\end{lemma}
\begin{proof}
Equation \eqref{eqiota} implies that~$\iota_{XY}(\cJ(\unit,X^\vee\otimes Y))\subset \cJ(X,Y)$. Similarly, equation~\eqref{eqiota1} implies that~$\iota^{-1}_{XY}(\cJ(X,Y))\subset\cJ(\unit,X^\vee\otimes Y) $.\end{proof}

\begin{proof}[Proof of Proposition~\ref{PropPhi}] For~$\Mmod\in \Sub_{\cC}(\Pmod_{\unit})$ and $X,Y\in\Ob\cC$,
take $h\in \iota_{XY}(\Mmod(X^\vee\otimes Y))$ and arbitrary~$f\in\cC(W,X)$ and~$g\in \cC(Y,Z)$. The fact that~$\Mmod$ is a submodule of~$\cC(\unit,-)$ and Lemma~\ref{Lemfg} imply that
$$g\circ h\in \iota_{XZ}(\Mmod(X^\vee\otimes Z))=\cJ_{\Mmod}(X,Z)\quad\mbox{and}\quad h\circ f\in\iota_{WZ}(\Mmod(W^\vee\otimes Y))=\cJ_{\Mmod}(W,Y),$$
which shows that~$\cJ_{\Mmod}$ is an ideal.
Lemma~\ref{LemTenZ} then shows that~$\cJ_{\Mmod}$ is a left-tensor ideal. This concludes the proof of part~(i).

That $\Phi\circ\Psi$ is the identity is precisely Lemma~\ref{LemJJ}. That $\Psi\circ\Phi$ is the identity follows from construction. This proves part~(ii).
\end{proof}

\begin{proof}[Proof of Theorem~\ref{Thm1}]
By Proposition~\ref{PropPhi}, it suffices to prove that~$\Psi$ respects the partial orders, which is trivial.
\end{proof}

\begin{proof}[Proof of Lemma~\ref{PropAK}]
Since $\tr(g\circ f)=\tr(f\circ g)$, see \ref{Condtr}, it follows that the negligible morphisms constitute an ideal $\cN$ in $\cC$, in the sense of \ref{IdPA}. 
Now consider arbitrary $X,Y,Z\in\Ob\cC$ and morphisms $f\in\cC(X,Y)$ and $g\in\cC(Z\otimes Y,Z\otimes X)$.
Naturality of the braiding and the hexagon identities yield
$$\tr(g\circ (1_Z\otimes f))\;=\;\tr( \theta \circ f),$$
with $\theta\in\cC(Y,X)$ given by
$$ (\ev_Z\otimes 1_X)\circ(1_Z\otimes\gamma_{XZ^\vee}\circ\gamma_{Z^\vee X})\circ (\gamma_{Z^\vee Z}\otimes 1_X)\circ (1_{Z^\vee}\otimes g)\circ (\co_Z\otimes 1_Y).$$
Consequently, $\cN$ is actually a tensor ideal.
It follows easily that, by assumption~(III), a morphism $f:\unit\to X$ is negligible if and only if $g\circ f=0$ for all $g\in \cC(X,\unit)$. In particular, $\Psi(\cN)$ is the unique maximal submodule of $\Pmod_{\unit}$. The result now follows from Theorem~\ref{Thm1}.
\end{proof}

\subsection{Example 1: Temperley-Lieb category} In~\cite[Definition~2.1]{GLTL}, the Temperley-Lieb category~$\cT=\cT_{K,q}$ was introduced, for any commutative unital ring $K$ and invertible~$q\in K$.
\subsubsection{}\label{TLdiag}
 The category~$\cT$ is $K$-linear skeletal (and hence strict) monoidal, with~$\Ob\cT\,=\,\mN$. The $K$-module~$\cT(i,j)$ is the free $K$-module spanned by non-intersecting planar diagrams between $i$ and~$j$ dots placed on two horizontal lines. The diagrams in~$\cT(4,2)$ are for instance given by
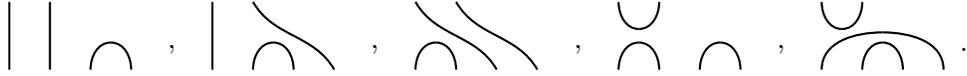
\begin{equation*}
\begin{tikzpicture}[scale=0.9,thick,>=angle 90]

\draw (0,0) to   [out=90,in=-90] +(0,1);
\draw (0.6,0) to   [out=90,in=-90] +(0,1);
\draw (1.2,0) to [out=90,in=-180] +(.3,.4);
\draw (1.5,.4) to [out=0,in=90] +(.3,-.4);

\node at (2.4,0.3) {,};

\draw (3,0) to   [out=90,in=-90] +(0,1);
\draw (3.6,0) to [out=90,in=-180] +(.3,.4);
\draw (3.9,.4) to [out=0,in=90] +(.3,-.4);
\draw (4.8,0) to   [out=120,in=-60] +(-1.2,1);
\node at (5.4,0.3) {,};

\draw (6,0) to [out=90,in=-180] +(.3,.4);
\draw (6.3,.4) to [out=0,in=90] +(.3,-.4);
\draw (7.2,0) to   [out=120,in=-60] +(-1.2,1);
\draw (7.8,0) to   [out=120,in=-60] +(-1.2,1);
\node at (8.4,0.3) {,};

\draw (9,0) to [out=90,in=-180] +(.3,.4);
\draw (9.3,.4) to [out=0,in=90] +(.3,-.4);
\draw (9,1) to [out=-90,in=-180] +(.3,-.4);
\draw (9.3,.6) to [out=0,in=-90] +(.3,.4);
\draw (10.2,0) to [out=90,in=-180] +(.3,.4);
\draw (10.5,.4) to [out=0,in=90] +(.3,-.4);

\node at (11.4,0.3) {,};

\draw (12,1) to [out=-90,in=-180] +(.3,-.4);
\draw (12.3,.6) to [out=0,in=-90] +(.3,.4);
\draw (12,0) to [out=90,in=-180] +(.9,.55);
\draw (12.6,0) to [out=90,in=-180] +(.3,.4);
\draw (12.9,.4) to [out=0,in=90] +(.3,-.4);
\draw (12.9,0.55) to [out=0,in=90] +(.9,-.55);
\node at (14.1,0.3) {.};

\end{tikzpicture}
\end{equation*}
Composition of diagrams is given by concatenation and evaluation of loops at $\delta:=-q-q^{-1}$. In particular, the Temperley-Lieb algebra $\TL_n(\delta)$, in terms of the Kauffman diagram calculus, is given by~$\cT(n,n)$.

The lines which connect dots on different horizontal lines are known as {\bf propagating} lines, the lines connecting dots on the upper horizontal line are {\bf cups} and those connecting dots on the lower line are {\bf caps}.
Inside~$\cT$, we have $i\otimes j=i+j$, for~$i,j\in\mN$ and~$d_1\otimes d_2$, for two diagrams $d_1,d_2$, is given by juxtaposition. That $\cT$ admits a braiding follows e.g. from the equivalence mentioned in~\ref{RemTL}(iii).  Furthermore, $\cT$ is rigid, with~$i^\vee=i$, for all~$i\in\mN$. It follows that~$\cT$ satisfies (I), (IV) and (V).  As a monoidal category, $\cT$ is generated by~$1\in\Ob\cT$ and the diagrams $I\in \cT(1,1)$, $\cup\in \cT(0,2)$ and~$\cap\in\cT(2,0)$. {\em We stress that~$1\not=\unit_{\cT}=0$.}
 
 \subsubsection{}\label{ResGL} From now on, we assume that~$K=\mk$ is an algebraically closed field of {characteristic zero}. In~\cite[Definition~2.2]{GLTL}, the cell module~$\Wmod_i\in \cT-$Mod is introduced for~$i\in\mN$, as the submodule of~$\Pmod_i=\cT(i,-)$ such that~$\Wmod_i(n)$ is spanned by all diagrams which do not contain caps. In particular, $\Wmod_0\simeq \Pmod_{\unit}$ and~$\Wmod_i(j)=0$ if $j<i$. It then follows from \cite[Theorem~5.3]{GLTL} that~$\Pmod_{\unit}=\cT(0,-)$ is simple unless $q^2$ has finite order~$l>1$. Hence $\cT$ has no proper tensor ideals unless $q^2$ has finite order~$l>1$. In the latter case, the unique proper submodule~$\Mmod$ of~$\Pmod_{\unit}$ is a homomorphic image of~$\Wmod_{2l-2}$. For instance, when $q^2=-1$, we have $l=2$ and the proper submodule~$\Mmod$ of~$\Pmod_{\unit}$, satisfies~$\Mmod(i)=\Pmod_{\unit}(i)$, for all~$i>0$.

 \begin{cor}\label{CorTL}
Take $q\in\mk^\times$ such that~$q^2$ has finite order~$l>1$.
\begin{enumerate}[(i)]
\item The Temperley-Lieb category $\cT$ has exactly one proper tensor ideal, the ideal of negligible morphisms~$\cJ$. The ideal $\cJ$ is generated by a quasi-idempotent $f\in \cT(l-1,l-1)$. 
\item The quasi-idempotent $f\in \TL_{l-1}(\delta)$ satisfies~$fd=0=d f$ for any diagram~$d\in \TL_{l-1}(\delta)$ which contains a cup or cap, so is in particular central.
\end{enumerate}
\end{cor}
\begin{proof}
It follows from \ref{ResGL} and Theorem~\ref{Thm1} that~$\cT$ contains exactly one proper tensor ideal~$\cJ$ and that~$\cJ(l-1,l-1)\simeq \cJ(0,2l-2)$ is non-zero. That $\cJ$ is the ideal of negligible morphisms follows from Lemma~\ref{PropAK}.

By uniqueness, $\cJ$ is generated by an arbitrary non-zero $f\in \cJ(l-1,l-1)$. If~$f$ would not be annihilated by diagrams with cups and caps, then 
$$\cJ(l-3,l-1)\;\simeq \;\cJ(l-1,l-3)\;\simeq\;\cJ(0,2l-4)\;\not=\; 0,$$
where the isomorphisms are given by Lemma~\ref{LemJJ}, which is contradicted by~\ref{ResGL}. 

We have $f=\alpha +g$, where $\alpha\in\mk$ represents~$\alpha$ times the diagram in~$\TL_{l-1}$ with only propagating lines, and~$g$ is in the ideal in~$\TL_{l-1}$ spanned by all diagrams with cups and caps. Since we have $fg=0=gf$ it follows that~$f^2=\alpha f$. 
\end{proof}

\begin{rem}\label{RemTL}Keep the assumptions as in Corollary~\ref{CorTL}.
\begin{enumerate}[(i)]
\item The classification of tensor ideals in~$\cT$ was proved by Goodman - Wenzl in~\cite{GW}.
\item The quasi-idempotent $f$ in Corollary~\ref{CorTL}(ii) can be normalised to an idempotent $e\in \TL_{l-1}$, known as the {\em Jones-Wenzl idempotent}.
\item By e.g. the proof of~\cite[Theorem~2.4]{Ostrik}, the pseudo-abelian envelope $\cT^{\oplus\sharp}$ is equivalent to the category of tilting modules over~$U_q(\mathfrak{sl}_2)$. For~$q$ a root of unity, by uniqueness of~$\cJ$, the quotient $(\cT/\cJ)^{{\oplus\sharp}}$ is {\em Andersen's fusion category} of~\cite{Andersen}. 
\end{enumerate}
\end{rem}

\subsection{Example 2: Deligne's category~$\RS$}\label{SecRS}
Let~$\mk$ be an algebraically closed field of characteristic zero and~$\cC$ the monoidal category~$\RS$, for some~$t\in \mk$, see \cite[Section~2]{Deligne} or \cite{ComesOst}. Conditions (I)-(V) are satisfied. We have $\cC=\cC_0^{\oplus\sharp}$, for~$\cC_0$ the {\em partition category} with~$\Ob\cC_0=\mN$, and~$\unit=0$, see e.g. \cite[Section~8]{Borelic}. When $t\not\in\mN$, \cite[Th\'eor\`eme~2.18]{Deligne} states that~$\RS$ is abelian semisimple. We thus have no non-trivial tensor ideals. The case $t\in\mZ_{>0}$ can easily be dealt with using Theorem~\ref{Thm1}.

\begin{prop}\label{PropCO}
When $t\in \mZ_{>0}$, the category~$\RS$ has a unique proper tensor ideal, the ideal of negligible morphisms.
\end{prop}
\begin{proof}
By Theorem~\ref{Thm1} and Lemma~\ref{PropAK}, it suffices to prove that~$\Pmod^{\cC}_{\unit}$, or equivalently, $\Pmod^{\cC_0}_{\unit}$, has precisely one proper submodule. The latter property is equivalent to the property that there exists~$k_0\in\mN$ such that the module~$\oplus_{i\le k}\cC_0(0,i)$ over the $\mk$-algebra $\oplus_{i,j\le k}\cC_0(i,j)$ has precisely one proper submodule, for all~$k>k_0$.

By the Morita equivalence in~\cite[Theorem~8.5.1]{Borelic} (and the definition of cell modules in~\cite[Proposition~8.6.4]{Borelic}), it suffices to prove that the cell module~$W_k(\varnothing)=\cC_0(0,k)$ of the partition algebra $P_k(t)=\cC_0(k,k)$, see \cite{Jones2}, has length two for all~$k>>0$. The latter is proved in~\cite[Lemma~5.11]{King}. In fact, the unique proper submodule of~$W_k(\varnothing)$ is $L_k((t+1))$, when $k>t$.
\end{proof}

\begin{rem}
Proposition~\ref{PropCO} was first proved by Comes - Ostrik in~\cite[Theorem~3.15]{ComesOst}, as an important step towards their proof of~\cite[Conjecture~8.21]{Deligne}.
\end{rem}

\begin{cor}\label{CorSt}
When $t\in \mZ_{>0}$, the only dense full monoidal functor, excluding equivalences, from $\RS$ is, up to isomorphism, given by the functor in~\cite[Th\'eor\`eme~6.2]{Deligne}:
$$\Fmod:\;\RS\,\to\, \Rep_{\mk}{\mathrm S}_t.$$
 
\end{cor}

\subsection{Example 3: Tilting modules for quantum groups}
Let $\fg$ be a simple Lie algebra over~$\mC$ of type ADE, with $h$ the Coxeter number, and fix $l\in\mZ_{>h}$. 

\subsubsection{The affine Kac-Moody algebra} We have the central extension
$$\tilde{\fg}\;=\;\fg\otimes\mC[t,t^{-1}]\oplus \mC K$$
of the loop algebra of $\fg$, where we normalise the central element $K$ such that
$$[X\otimes t^k,Y\otimes t^l]\;=\; [X,Y]\otimes t^{k+l}+\frac{\langle X,Y\rangle}{2h} k\delta_{k,-l}K,\quad\mbox{for $X,Y\in\fg$ and $k,l\in\mZ$.}$$
Here $\langle X,Y\rangle$ denotes the Killing form.
The affine Kac-Moody algebra is then given by
$$\hat{\fg}\;=\;\tilde{\fg}\oplus \mC\partial,\quad\mbox{with $\partial=t\partial_t$}. $$
We have the induced $\hat\fg$-module
$$\Delta_l:=U(\hat{\fg})\otimes_{U(\hat{\fg}^+)} \mC_l,\quad \mbox{with}\quad\widehat{\fg}^+\;:=\;\fg\otimes\mC[t]\oplus\mC K\oplus\mC\partial$$
where $\fg\otimes\mC[t]\oplus\mC\partial$ acts trivially on $\mC_l$, while $K$ acts through $l-h$.

\subsubsection{The quantum group} We set $q:=e^{-\frac{\pi \imath}{l}}\in\mC$ and denote by $U_q(\fg)$ Lusztig's version (with divided powers) of the quantum group corresponding to $\fg$, see \cite[II.H.6]{Jantzen}. We denote by $\Ti(U_q(\fg))$ the monoidal category of tilting modules for $U_q(\fg)$, see~\cite[II.H.15]{Jantzen}. It satisfies conditions (I)-(V).

\begin{thm}
We have an isomorphism between the lattice of tensor ideals in $\Ti(U_q(\fg))$ and the lattice of submodules of the $\hat{\fg}$-module~$\Delta_l$.
\end{thm}
\begin{proof}
Set $\cC:=\Ti(U_q(\fg))$. For any $c\in\mC$, denote by $\bO_c$ the category of all $\hat{\fg}$-modules which are locally finite for $\hat{\fg}^+$ and semisimple for $\hat{\fg}_0:=\fg\oplus\mC K\oplus\mC\partial$ such that $K$ acts through $c-h$. 

It follows from \cite[Theorem~6.1, Bemerkung~6.5(2)]{Soergel} and \cite[Theorem~2]{KK} that each simple module in $\bO_l$ has a projective cover and $\Delta_l$ is such a cover. Denote by $\bP$ the full subcategory of $\bO_l$ consisting of projective covers of simple modules. 
Consider the functor $\bP(-,\Delta_l):\bP^{\op}\to \Ab$. We have an isomorphism of lattices
$$\Sub_{U(\hat{\fg})}(\Delta_l)\;\stackrel{\sim}{\to}\;\Sub_{\bP^{\op}}\bP(-,\Delta_l),\qquad M \mapsto \bO_l(-,M),$$
where for each submodule~$M\subset \Delta_l$ we regard 
$$\bO_l(-,M):\bP^{\op}\to\Ab$$ as a subfunctor of $\bP(-,\Delta_l)=\bO_l(-,\Delta_l)|_{\bP^{\op}}$ in the canonical way. Since the full subcategory of $\bO_l$ of modules with finite dimensional weight spaces has a simple preserving duality and that category includes $\bP$, we have $\bP\simeq\bP^{\op}$. Hence we find an isomorphism 
$$\Sub_{U(\hat{\fg})}(\Delta_l)\;\stackrel{\sim}{\to}\;\Sub_{\bP}\Pmod_{\Delta_l}.$$

It is proved in \cite[Section~6]{Soergel}, that we have an equivalence $\bP\stackrel{\sim}{\to}\bT$, where $\bT$ is the category of indecomposable tilting modules in $\bO_{-l}$. By a result of Polo, see \cite[Proposition~8.1]{Soergel}, forgetting the action of $\partial$ on $\bO_{-l}$ yields an equivalence with a similarly defined category of $\tilde{\fg}$-modules. The full subcategory of all modules with finite length in the latter category was studied in \cite{KL} where it was proved to be equivalent to the category of finite dimensional modules of type 1 over~$U_q(\fg)$. In particular, this restricts to an equivalence $\bT^{\oplus}\stackrel{\sim}{\to}\cC$. Tracing the module~$\Delta_l$ through all equivalences shows that it gets sent to $\unit$ in $\cC$. In particular this implies that
$$\Sub_{\bP}\Pmod_{\Delta_l}\;\stackrel{\sim}{\to}\; \Sub_{\cC}(\Pmod_{\unit}).$$
The conclusion now follows from Theorem~\ref{Thm1}.
\end{proof}

\begin{rem}\label{Rem3}
The thick tensor Ob-ideals in $\Ti(U_q(\fg))$ have been classified in \cite{OstrikT}. It follows that $\TId(\cC)$ will contain more elements than $\Ide([\cC]_{\oplus})$. For example, for $U_q(\mathfrak{sl}_3)$ we have 3 ideals in $\Ide([\cC]_{\oplus})$, whereas one can calculate, using Kazhdan-Lusztig combinatorics, that $\Delta_l$ has at least 5 simple constituents (private communication with Michael Ehrig).  
\end{rem}


\section{Fibres of the decategorification map}\label{SecFibres}
Fix a monoidal category~$\cC$ satisfying (I) and (II).
\subsection{The decategorification map}

\begin{ddef}\label{DefOb}
For any left-tensor ideal $\cJ$ in~$\cC$, we define the set
$$\Ob(\cJ)\:=\;\{X\in\Ob\cC\,|\, 1_X\in \cJ(X,X)\}\;=\; \{X\in\Ob\cC\,|\, \cJ(X,X)=\cC(X,X)\}.$$
\end{ddef}

Then $\Ob(\cJ)$ is a thick left-tensor Ob-ideal in~$\cC$ and we have the corresponding map
$$\Ob\;:\; \TId(\cC)\to \Ide([\cC]_{\oplus}).$$ The notation is justified by the observation that 
$$\Ob(\cJ)\;=\;\{X\in\Ob\cC\,|\, X\simeq \zero\;\,\mbox{ in }\;\,\cC/\cJ\}.$$
Furthermore, we have a group isomorphism
$[\cC/\cJ]_{\oplus}\simeq [\cC]_{\oplus}/I,$
with~$I$ the thick left ideal, in the based ring $[\cC]_{\oplus}$, associated to $\Ob(\cJ)$.
This is a ring isomorphism when $\cC$ is braided. 
Motivated by these observations we refer to~$\Ob(-)$ as the {\bf decategorification map}.

Proposition~\ref{PropPhi} implies
\begin{equation}\label{ObPsi}
\Ob(\Psi^{-1}(\Mmod))=\{X\in\Ob\cC\,|\, \co_X\in \Mmod(X^\vee\otimes X)\}.
\end{equation}
\begin{thm}\label{ThmFibre}
Let~$\cC$ be a right rigid Krull-Schmidt monoidal category.
\begin{enumerate}[(i)]
\item We have a surjective morphism of partially ordered sets
$$\Ob:\TId(\cC)\tto \Ide([\cC]_{\oplus}).$$
\item 
For~$\sI\in \Ide([\cC]_{\oplus})$, the minimal element in the fibre $\Ob^{-1}(\sI)$ is given by the tensor ideal
$$\cJ^{\min}_{\sI}(X,Y)\;=\;\{f\in\cC(X,Y)\,|\, \mbox{there exists~$Z\in\sI$ such that~$f$ factors as $X\to Z\to Y$}\}.$$
\item The minimal element in~$\Psi\left(\Ob^{-1}(\sI)\right)\subset\Sub(\Pmod_{\unit})$ is given by~$\Tr_{\sI}\Pmod_{\unit}$.
\end{enumerate}
\end{thm}
\begin{proof}
We start by considering an arbitrary~$\sI\in \Ide([\cC]_{\oplus})$.
By construction, $\cJ^{\min}_{\sI}$ is an ideal in~$\cC$ and the fact that~$\sI$ is a thick tensor~$\Ob$-ideal easily shows that~$\cJ^{\min}_{\sI}$ is a tensor ideal.

Take $X\in \Ob\cC$ with~$1_X\in\cJ^{\min}_{\sI}(X,X)$. Then there exist $Z\in \sI$, $f\in \cC(X,Z)$ and~$g\in \cC(Z,X)$ such that~$1_X=g\circ f$. Since $\cC$ is karoubian, this means that there exists~$Y\in\cC$ such that~$Z\simeq X\oplus Y$. Since $\sI$ is thick, we find~$X\in\sI$. On the other hand, if follows by definition that~$\sI\subset \Ob(\cJ^{\min}_{\sI})$. In conclusion, we find $\Ob(\cJ^{\min}_{\sI})=\sI$.

Now we can prove part~(i). It is obvious that~$\Ob(-)$ is a morphism of partially ordered sets. That $\Ob(-)$ is surjective then follows from the conclusion of the above paragraph.

Now we prove part~(ii). For~$\sI\in \Ide([\cC]_{\oplus})$, any ideal in~$\Ob^{-1}(\sI)$ must contain~$1_Z$, for all~$Z\in \sI$. By construction, $\cJ^{\min}_{\sI}$ is thus minimal in the fibre over~$\sI$. 

Comparing with equation~\eqref{eqTr}, while using the additivity of~$\cC$ (and $\sI$), shows that~$\cJ^{\min}_{\sI}(\unit, Y)=\Tr_{\sI}\Pmod_{\unit}(Y)$, for all~$Y\in\Ob\cC$. This concludes the proof of part~(iii).
\end{proof}

\subsection{Obstructions to~$\Ob(-)$ being an isomorphism}
In Theorem~\ref{ThmSuff}, we will present sufficient conditions for the surjective map $\Ob(-)$ to be a bijection. In this section, we demonstrate that these conditions are close to necessary.

\subsubsection{The set~$\sB$}\label{SetB} It will be convenient to introduce the set
$$\sB=\{X\in\inde\cC\,|\, \Pmod_{\unit}^{\cC}(X)=\cC(\unit,X)\not=0\}.$$
Note that all indecomposable direct summands of~$\unit$ are in~$\sB$. By the Yoneda lemma, $\sB$ corresponds to the set of~$X\in\inde\cC$ for which
$\Nat(\Pmod_{X},\Pmod_{\unit})$ is not zero.
In particular, for any~$\sS\subset\Ob\cC$, we have $\Tr_{\sS}\Pmod_{\unit}=\Tr_{\sS\cap\sB}\Pmod_{\unit}.$

\begin{prop}\label{PropMult}
Assume that there exists~$X\in\sB$ such that $\cC(\unit,X)$ is not a simple~$\cC(X,X)$-module, then $\Pmod_{\unit}$ contains submodules which are not trace submodules.

Consequently, $\Ob(-):\TId(\cC)\tto \Ide([\cC]_{\oplus})$ is not an isomorphism. \end{prop}
\begin{proof}
Since $\cC(\unit,X)$ is not simple, we have a non-zero $f\in \cC(\unit,X)$ such that~$\cC(X,X)\circ f\not=\cC(\unit,X)$. Consider the submodule~$\Mmod$ of~$\Pmod_{\unit}$ generated by $f$.
By Lemma~\ref{Lem2Tr}, $\Mmod$ can only be a trace submodule $\Tr_{\sS}\Pmod_{\unit}$ if $X\in\sS$. However, already $\Tr_X\Pmod_{\unit}$ is strictly bigger than $\Mmod$, so $\Mmod$ is not a trace submodule.

That $\Ob(-)$ cannot be an isomorphism follows from Theorem~\ref{ThmFibre}(iii), which shows that there must be a fibre of~$\Ob:\TId(\cC)\tto \Ide([\cC_{\oplus}])$ which contains more than one element. 
\end{proof}

\begin{cor}\label{CorNoIso}
Assume that $\cC$ satisfies (I)-(IV) with $\mk:=\cC(\unit,\unit)$ algebraically closed.
If there exists~$X\in \sB$ with
$\dim_{\mk}\cC(\unit,X)>1$, then $\Ob(-):\TId(\cC)\tto \Ide([\cC]_{\oplus})$ is not an isomorphism. \end{cor}
\begin{proof}
Under the assumption in the corollary, all simple $\cC(X,X)$-modules are one-dimensional. The conclusion follows from Proposition~\ref{PropMult}.
\end{proof}

Another form of obstruction is discussed in the following lemma.
\begin{lemma}\label{LemNec}
Assume that there exists~$Z\in \sB$, such that there is no $X\in \Ob\cC$ and $k\in\mN$ for which $\co_X$ factors as $\unit\to Z^{\oplus k}\to X^\vee\otimes X$. Then the morphism $\Ob(-):\TId(\cC)\tto \Ide([\cC]_{\oplus})$ is not an isomorphism.
\end{lemma}
\begin{proof}
Let~$\Mmod$ be the submodule~$\Tr_{Z}\Pmod_{\unit}$ of~$\Pmod_{\unit}$. 
Equation~\eqref{ObPsi} implies that $0\not=\cJ:=\Psi^{-1}(\Mmod)\in \TId(\cC)$ satisfies~$\Ob(\cJ)=\emptyset$.
\end{proof}

The observations in this section justify looking at categories with properties as in the following lemmata.

\begin{lemma}\label{LemOnlyTrace}
Assume that the left $\cC(X,X)$-module
$\cC(\unit,X)$ is simple for every~$X\in \sB$, then every submodule of~$\Pmod_{\unit}$ is a trace submodule.
\end{lemma}
\begin{proof}
For any submodule~$\Mmod$ of~$\Pmod_{\unit}$ and~$Y\in \sB$ we have either $\Mmod(Y)=0$ or~$\Mmod(Y)=\cC(\unit,Y)$. Let~$\sS$ be the subset of~$Y\in\sB$ for which $\Mmod(Y)=\cC(\unit,Y)$. Then we have
$\Mmod\;=\;\Tr_{\sS}\Pmod_{\unit}.$
\end{proof}

\begin{lemma}\label{LemXZ1}
If, for~$Z\in\sB$ and $X\in\inde\cC$, we have that~$\co_{X}$ is a composition~$\unit\to Z^{\oplus k}\to X^\vee\otimes X$, for some~$k\in\mZ_{>0}$, then $X\inplus X\otimes Z$.
\end{lemma}
\begin{proof}
It follows from equation~\eqref{eqdual} that~$1_X$ factors through $X\to X\otimes Z^{\oplus k}\to X$, meaning $X\inplus X\otimes Z^{\oplus k}$. 
\end{proof}

\begin{lemma}\label{LemXZ}
If, for~$Z\in\sB$, there exists~$X\in\inde\cC$, such that 
$\add(X^\vee\otimes X)\cap\sB=\{Z\},$ then
\begin{enumerate}[(i)]
\item $\co_{X}$ is a composition~$\unit\to Z^{\oplus k}\to X^\vee\otimes X$, for some~$k\in\mZ_{>0}$;
\item $\Psi^{-1}(\Tr_{Z}\Pmod_{\unit})$ is generated by $1_X$;
\item $\Ob(\Psi^{-1}(\Tr_{Z}\Pmod_{\unit}))$ is generated by $X$;
\item $\Ob(\Psi^{-1}(\Tr_{Z}\Pmod_{\unit}))$ is generated by $Z$.
\end{enumerate}
\end{lemma}
\begin{proof}
By assumption, $Z$ is the only indecomposable direct summand of~$X^\vee\otimes X$ which admits a non-zero morphism $\unit\to Z$. Hence $\co_X$ must factor as stated in part (i).

By assumption and part (i), $\Tr_{Z}\Pmod_{\unit}$ is generated, as a submodule of $\Pmod_{\unit}$, by $\co_X$. It follows that $\Psi^{-1}(\Tr_{Z}\Pmod_{\unit})$ is the minimal ideal containing $\co_X$. By equation~\eqref{eqdual}, any tensor ideal containing $\co_X$ also contains $1_X$ and the reverse is clearly also true, proving part (ii). Part (iii) then follows immediately.

By assumption, $Z$ is contained in the left tensor Ob-ideal generated by $X$. By Lemma~\ref{LemXZ1}, $X$ is contained in the ideal generated by $Z$. Part (iv) thus follows from part (iii).\end{proof}

\subsection{Main theorem}

\begin{thm}\label{ThmSuff}
Consider a monoidal category~$\cC$ satisfying (I) and (II). Assume that for each $Z\in \sB$ (with~$\sB$ as in~\ref{SetB})
\begin{enumerate}[(a)]
\item the $\cC(Z,Z)$-module~$\cC(\unit,Z)$ is simple.
\item there exists~$X_Z\in\inde\cC$, such that  $\add(X_Z^\vee\otimes X_Z)\cap\sB=\{Z\}.$
\end{enumerate}
 Then $\Ob(-): \TId(\cC)\to \Ide([\cC]_\oplus)$ is an isomorphism of partially ordered sets.
\end{thm}

\begin{proof}
By Theorems~\ref{ThmFibre}(iii) and \ref{Thm1}, it suffices to prove that every submodule of~$\Pmod_{\unit}$ is of the form 
$\Tr_{\sJ}\Pmod_{\unit}$, for some~$\sJ\in \Ide([\cC]_{\oplus})$. By Lemma~\ref{LemOnlyTrace}, we already know that all submodules are trace submodules.

Since $\Tr_{\sS}\Pmod_{\unit}=\Tr_{\sS\cap\sB}\Pmod_{\unit}$ for all $\sS\subset\inde\cC$, we take an arbitrary subset~$\sS\subset\sB$, consider the trace submodule~$\Mmod:=\Tr_{\sS}\Pmod_{\unit}$ and denote by $\sI\in \Ide([\cC]_{\oplus})$ the ideal generated by $\sS$. We will prove that $\Tr_{\sS}\Pmod_{\unit}=\Tr_{\sI}\Pmod_{\unit}$.
We take $Z\in \sI\cap \sB$ and write~$X:=X_{Z}$. By assumption, there exists $Z_0\in\sS$ and $Y'\in \Ob\cC$ such that~$Z\inplus Y'\otimes Z_0$. We set $X_0:=X_{Z_0}$, so we have in particular $Z_0\inplus X_0^\vee\otimes X_0$.  By Lemma~\ref{LemXZ1}, we have $X\inplus X\otimes Z$. These three observations imply that~$X\inplus Y\otimes X_0$, with~$Y:=X\otimes Y'\otimes X_0^\vee$. Since $Z\inplus X^\vee\otimes X$, we thus have
\begin{equation}\label{eqZXYYX}Z\inplus \;X_0^\vee\otimes Y^\vee\otimes Y\otimes X_0.\end{equation}

Lemma~\ref{LemXZ}(i) implies that~$$\co_{X_0}\,\in\, \Tr_{Z_0}\Pmod_{\unit}(X_0^\vee\otimes X_0)\,\subset \,\Mmod(X_0^\vee\otimes X_0).$$ Equation~\eqref{eqXYd} then implies
$$\co_{Y\otimes X_0}\,\in \,\Mmod(X^\vee_0\otimes Y^\vee\otimes Y\otimes X_0).$$
Since $\cC(\unit, W^\vee\otimes W)$ is always generated by $\co_W$, see \ref{secdual}, it follows that the latter space is equal to~$\cC(\unit,X^\vee_0\otimes Y^\vee\otimes Y\otimes X_0)$. 
By~\eqref{eqZXYYX}, we thus have $\Mmod(Z)=\cC(\unit,Z)$. This implies indeed that~$\Mmod=\Tr_{\sI}\Pmod_\unit$.\end{proof}

By Proposition~\ref{PropMult}, condition \ref{ThmSuff}(a) is necessary for~$\Ob$ to be bijective. However, it is in itself ({\it i.e.} without for instance condition~(b)) not sufficient, by the following example.

\begin{ex}
Let $p\in\mN$ be a prime. Consider the cyclic group $G:=\mZ/p\mZ=\mF_p^+$ and a field~$\mk$ with $\charr(\mk)=p$. We have the (abelian) monoidal category $\cC:=\Rep_{\mk}G$, satisfying (I)-(V), of finite dimensional modules of the Hopf algebra $\mk G$. Note that $\mk G$ can be identified with an infinitesimal group scheme. The indecomposable modules can be labelled by their dimension as $\{M_i\,|\,1\le i\le p\}$. It is easily calculated that $\Ide([\cC]_{\oplus})$ contains, besides $\Ob\cC$ and $\emptyset$, only the ideal of projective modules (direct sums of $M_p=\mk G$). As a side comment, we note that this is consistent with $|\Spec( \Stab \mk G)|=1$, see \cite[Theorem~6.3(b)]{Balmer}. On the other hand, we have $\dim_{\mk}\cC(\unit,M_i)=1$, for all~$i$, leading to the conclusion that \ref{ThmSuff}(a) is satisfied and that we have $p+1$ tensor ideals by Theorem~\ref{Thm1}. Hence $|\TId(\cC)|>|\Ide([\cC]_{\oplus})|$ if $p>2$. It follows easily that the $p-1$ proper ideals in $\TId(\cC)$ are mapped by $\Ob(-)$ to the unique proper ideal in $\Ide([\cC]_{\oplus})$.
\end{ex}

\begin{rem}
If~$\unit$ is indecomposable, and hence $\unit\in\sB$, the condition in Theorem~\ref{ThmSuff}(a) implies that~$\cC(\unit,\unit)$ has no left ideals. By \ref{DefMC}, it then follows that~$\cC(\unit,\unit)$ is a field. \end{rem}

Due to the above observation we note the following specific version of the main theorem.
\begin{thm}\label{ThmSuff2}
Consider a  monoidal category~$\cC$ satisfying (I)-(III) and set $\mk:=\cC(\unit,\unit)$. Assume that for each $Z\in \sB$
 there exists~$X_Z\in\inde\cC$, such that  
 $$Z\inplus X_Z^\vee\otimes X_Z\quad\mbox{ and }\quad\dim_{\mk}\cC(X_Z,X_Z)=1.$$
Then $\Ob(-): \TId(\cC)\to \Ide([\cC]_\oplus)$ is an isomorphism of partially ordered sets.
\end{thm}
\begin{proof}
By assumption, for every $Z\in\sB$ we have
$$0\;<\;\dim\cC(\unit,Z)\;\le\;\dim\cC(\unit, X^\vee_Z\otimes X_Z)\;=\;\dim\cC(X_Z,X_Z)\;=\;1.$$
Hence $\cC(\unit,Z)$ is one-dimensional, which implies \ref{ThmSuff}(a). Furthermore, this shows that $X^\vee_Z\otimes X_Z$ contains precisely one direct summand which is in $\sB$, so in particular \ref{ThmSuff}(b) is satisfied.
\end{proof}

\subsection{A special case}

\begin{ddef}\label{DefContr}
An {\bf $l$-controlled} monoidal category is a pair $(\cS,\ell)$ consisting of the following. A monoidal category $\cS$ satisfying (II) and (III). A map $\ell:\Lambda\to\mN$, for~$\Lambda:= \inde\cS$, such that $\ell^{-1}(0)=\{\unit\}$ and $\bkappa\inplus \bmu\otimes\blambda$ implies that $\ell(\bkappa)=\ell(\bmu)+\ell(\blambda)$, for all $\bkappa,\blambda,\bmu\in \Lambda$.

\end{ddef}
For~$(\cS,\ell) $ as in the definition, we define a binary relation $\preceq$ on $\Lambda$ by setting $\blambda\preceq\bnu$ if and only if $\bnu\inplus\bmu\otimes \blambda$, for some~$\bmu\in\Lambda$. Clearly we have
\begin{equation}\label{precle}\blambda\preceq\bnu\;\;\Rightarrow\; \mbox{ $\blambda=\bnu$ or~$\ell(\blambda)<\ell(\bnu)$.} \end{equation}
It follows that $\preceq$ is a partial order.

\subsubsection{}\label{tuple} Now we assume we have a tuple $(\cC,\Tmod,\cS,\ell)$ with
\begin{itemize}
\item $(\cS,\ell)$ an $l$-controlled monoidal category with $\Lambda:=\inde\cS$;
\item $\cC$ a monoidal category satisfying (I)-(IV) with $\mk:=\cC(\unit,\unit)$;
\item $\Tmod:\cS\to\cC$ a monoidal functor;
\end{itemize}
such that: 
\begin{enumerate}[(a)]
\item We have $\inde\cC=\{R(\blambda)\,|\,\blambda\in\Lambda\}\cong\Lambda$, with $\Tmod(\blambda)=R(\blambda)\oplus X$ for some~$X\in \Ob\cC$, where $R(\bmu)\inplus X$ implies~$\bmu\prec \blambda$.
\item There exists a map $\phi:\Lambda\to\Lambda$ with $\ell\circ\phi=\ell$, such that $R(\blambda)^\vee\simeq R(\phi(\blambda))$, for all $\blambda\in\Lambda$.
\item If~$\{R(\bnu),R(\bnu')\}\subset\sB$, then $\ell(\bnu)\not=\ell(\bnu')$.
Set $\mL:=\{0,1,\cdots,n\}$ if $n+1=|\sB|<\infty$ and $\mL=\mN$ otherwise. Then we label $\sB=\{R(\bnu^{(j)})\,|\, j\in\mL\}$, where $0=\ell(\bnu^{(0)})<\ell(\bnu^{(1)})<\ell(\bnu^{(2)})<\cdots$.
\item For each $j\in\mL$, the set
$\Lambda_j:=\{\blambda\in\Lambda\,|\, \bnu^{(j)}\inplus \phi(\blambda)\otimes \blambda\}$ is non-zero, and $\blambda\in\Lambda_j$ implies~$\cC(R(\blambda),R(\blambda))=\mk$.
\item  If~$j<j'$ for $j,j'\in\mL$, then for each $\blambda'\in\Lambda_{j'}$ there exists $\blambda\in\Lambda_j$ such that $\blambda\prec\blambda'$.
\item If for any $\blambda\in\Lambda$ and $j\in\mL$ we have $\bnu^{(j)}\preceq \bkappa\inplus \phi(\blambda)\otimes\blambda$, for some~$\bkappa\in\Lambda$, then there exists $\bmu\in \Lambda_j$ for which $\bmu\preceq\blambda$.
\end{enumerate}

\begin{lemma}\label{LemSC}
Consider a tuple $(\cC,\Tmod,\cS,\ell)$ as in~\ref{tuple} and take $\blambda,\bmu,\bkappa\in\Lambda$. 
\begin{enumerate}[(i)]
\item Assume $\ell(\bkappa)=\ell(\blambda)+\ell(\bmu)$. The number of times $R(\bkappa)$ appears as a direct summand in $R(\blambda)\otimes R(\bmu)$ is the same as the number of times $\bkappa$ appears as a direct summand in $\blambda\otimes\bmu$.
\item Assume that $\blambda\preceq\bmu$. For any $\sI\in \Ide([\cC]_{\oplus})$ with $R(\blambda)\in \sI$, we have $R(\bmu)\in\sI$.
\end{enumerate}
\end{lemma}
\begin{proof}
Part (i) follows from the fact that $\Tmod$ is a monoidal functor, equation~\eqref{precle} and assumption~(a). Part (ii) follows from part (i).
\end{proof}

\begin{thm}\label{ThmSC}
Consider a tuple $(\cC,\Tmod,\cS,\ell)$ as in~\ref{tuple}. 
\begin{enumerate}[(i)]
\item The decategorification map $\Ob:\Ide([\cC]_{\oplus})\to\TId(\cC)$ is an isomorphism. 
\item We have $\Ide([\cC]_{\oplus})=\{\sI_i\,|\,i\in\mL\}$ with $\Ob\cC=\sI_0\supsetneq\sI_1\supsetneq\sI_2\supsetneq\cdots$ and
$$ R(\blambda)\in\sI_i\quad\Leftrightarrow\quad \blambda\succeq \bmu\;\mbox{ for some~$\bmu\in\Lambda_i$}.$$
\item For~$i,j\in\mL$ and $\cJ_i:=\Ob^{-1}(\sI_i)$, we have
$$\cJ_i(\unit,R(\bnu^{(j)}))\;=\;\begin{cases}\cC(\unit,R(\bnu^{(j)}))\cong\mk&\mbox{if $i\le j$}\\
0&\mbox{if $i>j$.}\end{cases}
$$
\end{enumerate}
\end{thm}
\begin{proof}
For~$j\in\mL$, we take an arbitrary $\blambda\in \Lambda_j$. It follows from Lemma~\ref{LemSC}(i) that $R(\bnu^{(j)})$ appears as a direct summand of $R(\blambda)^\vee\otimes R(\blambda)$. By condition~(d), Theorem~\ref{ThmSuff2} thus implies part~(i). We also find $\cC(\unit,R(\bnu^{(j)}))\cong\mk$ by the proof of \ref{ThmSuff2}.

Now we define the trace submodule~$\Mmod_j:=\Tr_{R(\bnu^{(j)})}\Pmod_{\unit}$, for each $j\in\mL$. These are all distinct, by Lemma~\ref{Lem2Tr}. We study their behaviour under the isomorphism $\Ob\circ\Psi^{-1}$. By Lemma~\ref{LemXZ}(iii), $\sI_j:=\Ob\circ\Psi^{-1}(\Mmod_j)$ is the ideal generated by $R(\blambda)$, for an arbitrary $\blambda\in\Lambda_j$.
By condition~(e) and Lemma~\ref{LemSC}(ii), this implies that $\sI_j\subset\sI_{j'}$, so also $\Mmod_j\subset \Mmod_{j'}$ if $j>j'$. 

By Lemma~\ref{LemOnlyTrace} and the above paragraph, we find that $\{\Mmod_j\,|\, j\in\mL\}$ is an exhaustive list of submodules of $\Pmod_{\unit}$. Part~(iii) is now also clear. It remains to prove that
$$R(\blambda)\in \sI_j\quad\Leftrightarrow\quad\blambda\succeq \bmu\;\mbox{ for some~$\bmu\in\Lambda_j$}. $$
Since $R(\bmu)\in \sI_j$ for all $\bmu\in\Lambda_j$, Lemma~\ref{LemSC}(ii) implies it suffices to show $\Rightarrow$ above. By~\eqref{ObPsi} we can assume that $\co_{R(\blambda)}$ factors through $R(\bnu^{(j)})$. By part (iii) this means that $R(\bnu^{(j')})\inplus R(\blambda)^\vee\otimes R(\lambda)$ for some~$\bnu^{(j')}$ with $j'\ge j$. By conditions (a), (b) this means that $R(\bnu^{(j')})\inplus \Tmod (\phi(\blambda)\otimes\blambda)$. Consequently, we have $R(\bnu^{(j')})\inplus \Tmod(\bkappa)$, for some~$\bkappa\in\Lambda$ with $\bkappa\inplus \phi(\blambda)\otimes\blambda$. By condition~(a) we have $\bnu^{(j')}\preceq\bkappa$, which by condition~(f) means that there exists $\bmu'\in\Lambda_{j'}$ with $\blambda\succeq\bmu'$. Condition~(e) then provides $\bmu\in\Lambda_j$ with $\blambda\succeq\bmu$.
\end{proof}


\section{Example: the category of tilting modules for a reductive group}\label{SecModular}
In this section, we let $\mk$ be an algebraically closed field of positive characteristic $p$.
\subsection{Tilting modules}
 We consider a connected reductive algebraic group $G$ over~$\mk$, with maximal torus~$T$ and Borel subgroup $B\supset T$, see~\cite{Jantzen}.  We also assume that the derived group $G'$ is simple and simply connected. For instance, this excludes the multiplicative group $\mG_m$ with $\mG_m(\mk)=\mk^\times$, since $\mG_m'=0$. We denote the Coxeter number by $h$, see \cite[\S II.6.1]{Jantzen}. 

{\em We will always assume that $p\ge 2h-2$,} except when $G=\SL_2$.

\subsubsection{}  Let $\cC:=\Ti(G)$ be the monoidal category of $G$-modules with both a good and a Weyl filtration, known as tilting modules, see \cite{Donkin} or \cite[\S II.E]{Jantzen}. This is a monoidal subcategory of the category $\Rep_{\mk}G$ of all (finite dimensional) algebraic modules, by \cite[Proposition~1.2(i)]{Donkin} and the observation that the trivial module~$\unit$ is tilting. Clearly, $\cC$ satisfies (I)-(V).

We let $R^+$ denote the set of all positive roots.  We have the set $X^+\subset X(T)$ of dominant weights of \cite[\S II.2.6]{Jantzen}. Let $\rho\in X(T)\otimes_{\mZ}\mQ$ denote half the sum of positive roots, see \cite[\S II.1.6]{Jantzen}. For~$\lambda\in X^+$, we denote the induced module by $\nabla(\lambda)$ and the Weyl module by $\Delta(\lambda)$.

\subsubsection{}It is a well-known fact, as follows e.g. from \cite[Proposition~II.4.13]{Jantzen}, that
\begin{equation}\label{standardstuff}\dim_{\mk} \Hom_G(\Delta(\lambda),N)\;=\; (N:\nabla(\lambda)),\quad\mbox{for all $\lambda\in X^+,$}\end{equation}
for any $G$-module~$N$ with a good filtration, and $(N:\nabla(\lambda))$ the multiplicity of $\nabla(\lambda)$ in such a filtration. Since $\unit\simeq \Delta(0)$, this gives a description of the functor~$\Pmod_{\unit}=\cC(\unit,-)$.

By \cite[Theorem~1.1]{Donkin}, the indecomposable modules in $\Ti(G)$ are labelled as $T(\lambda)$ by $\lambda\in X^+$. The tilting module~$T(\lambda)$ can be characterised by the fact that there is a monomorphism $\Delta(\lambda)\hookrightarrow T(\lambda)$ such that the cokernel has a Weyl filtration.

We define the following dominant weights
$$\sigma_j:= (p^j-1)\rho\in X^+\quad\mbox{and}\quad \mu_j=2(p^j-1)\rho\in X^+,\quad\mbox{for }j\in\mN.$$
Here $\sigma_j$ is the Steinberg weight, for which we have the simple tilting module
$T(\sigma_j)\;\simeq\; L(\sigma_j),$ see e.g. \cite[II.3.19(4)]{Jantzen}.

\begin{lemma}\label{Lemmuj}
For each $j\in\mN$, we have $T(\mu_j)\inplus T(\sigma_j)^\vee\otimes T(\sigma_j)$ and
$$\dim_{\mk}\Hom_G(\unit,T(\mu_j))\;=\; 1\;=\;\dim_{\mk}\End_G(T(\sigma_j)).$$
\end{lemma}
\begin{proof}
First we observe that $T(\sigma_j)^\vee= T(\sigma_j)^\ast\simeq T(\sigma_j)$ by \cite[II.E.6(2)]{Jantzen}.
By \cite[Satz 6.2(3)]{JantzenC}, for the special case $n=j$, $\nu=0$ and $\lambda=\sigma_j$, we have an inclusion
$$\unit=L(0)\hookrightarrow \Delta(\mu_j)\hookrightarrow T(\mu_j).$$
That $T(\mu_j)$ is a direct summand of $T(\sigma_j)^{\otimes 2}$ follows immediately from $\mu_j=2\sigma_j$. Since $T(\sigma_j)$ is simple, the dimension formulae follow.
\end{proof}
 
\subsection{Tensor ideals and Frobenius kernels}\label{SecFK}
In this subsection we show how the above considerations lead to some conclusions also made by Andersen in \cite[\S 4]{AndersenProc}, based on results by Donkin in~\cite{Donkin}.

\begin{prop}\label{ModularProp}
The category $\Ti(G)$ has infinitely many tensor ideals 
$$\{\Psi^{-1}(\Tr_{T(\mu_j)}\Pmod_{\unit})\,|\,j\in\mN\}.$$
\end{prop}
\begin{proof}
By Lemma~\ref{Lemmuj}, we have non-zero submodules
$\Tr_{\Pmod_{T(\mu_j)}}\Pmod_{\unit}$ of the principal module~$\Pmod_{\unit}$, which are all distinct by Lemma~\ref{Lem2Tr}. The conclusion then follows from Theorem~\ref{Thm1}.
\end{proof}

\subsubsection{}\label{SecFrob} We can conceptually explain these tensor ideals by considering the $r$th Frobenius kernel~$G_r$ of $G$, for each $r\in\mZ_{>0}$. This is a finite (even infinitesimal) group scheme, see \cite[\S I.9]{Jantzen}. Restriction yields a monoidal functor~$\Rep_{\mk}G\to\Rep_{\mk}G_rT$. Taking the quotient of $\Rep_{\mk}G_rT$ with the tensor ideal of morphisms which factor through projective modules yields the stable module category $\Stab G_rT$. Via composition, we obtain a monoidal functor
$$\Fmod_r:\;\Ti(G)\;\to\; \Rep_{\mk}G \;\to\; \Rep_{\mk}G_rT\;\to\;\Stab G_rT,\quad\mbox{for all~$r\in\mZ_{>0}$.}$$

\begin{prop}\label{PropConc}
For each $r\in\mZ_{>0}$, set $\cJ_r:=\ker\Fmod_r$ and $\sI_r:=\Ob(\cJ_r)$. Then $\sI_r$ has as indecomposable objects $\{T(\lambda)\,|\, \lambda\in \sigma_r+X^+\}$.
Furthermore, $\Psi^{-1}(\Tr_{T(\mu_r)}\Pmod_{\unit})$ is the minimal ideal in $\Ob^{-1}(\sI_r)$.
\end{prop}
\begin{proof}
By~\cite[Lemma~II.E.8]{Jantzen}, we have $T(\lambda)\in \sI_r$ if and only if $\lambda\in \sigma_r+X^+$.

By Lemma~\ref{LemXZ}(ii), $\Psi^{-1}(\Tr_{T(\mu_r)}\Pmod_{\unit})$ is the minimal tensor ideal which contains $1_{T(\sigma_r)}$. Equivalently, $\Psi^{-1}(\Tr_{T(\mu_r)}\Pmod_{\unit})$ is the minimal ideal in the fibre of $\Ob$ over the thick tensor Ob-ideal generated by $T(\sigma_r)$.
Since $T(\sigma_r+\nu)$ is a direct summand of $T(\sigma_r)\otimes T(\nu)$, for all $\nu\in X^+$, it follows that the latter ideal is precisely $\sI_r$.
\end{proof}

\subsection{The rank one case}
In this section we classify tensor ideals for~$G=\SL_2$. We identify $X^+=\mN$, via $m\rho\mapsto m$.

\begin{thm}\label{ThmSL2}
Set $\cC=\Ti(\SL_2)$ with $p=\charr(\mk)>0$.
The decategorification map $\Ob:\TId(\cC)\tto \Ide([\cC]_{\oplus})$ is a bijection and $\Ide([\cC]_{\oplus})$ is given by $\{\sI_k\,|\, k\in\mN\}$, with
$$\Ob\cC=\sI_0\supsetneq \sI_1\supsetneq\sI_2\supsetneq \sI_3\supsetneq \cdots.$$
We have $T(m)\in \sI_k$ if and only if $m\ge p^k-1$.
\end{thm}

We precede the proof with a remark and a lemma.
\begin{rem}
\label{RemSL2}
Theorem~\ref{ThmSL2} states that the tensor ideals in $\Ti(\SL_2)$ are precisely the kernels of the functors $\Fmod_r$ in~\ref{SecFrob}.
\end{rem}

It is well-known that for~$\SL_2$ we can complete Lemma~\ref{Lemmuj} with the following lemma.
This is a
 consequence of Donkin's tensor product theorem, \cite[\S 2]{Donkin} or \cite[\S II.E.9]{Jantzen}, and equation~\eqref{standardstuff}.
 \begin{lemma}\label{LemLem}
For all~$m\in\mN$, we have
$$\dim_{\mk}\Hom_{\SL_2}(T(0),T(m))\;=\;\begin{cases}1&\mbox{if $m=2p^j-2$, for~$j\in\mN$,}\\
0&\mbox{otherwise.}\end{cases}$$
\end{lemma}

\begin{proof}[Proof of Theorem~\ref{ThmSL2}]
We set $T_j:=T(2p^j-2)$, for all~$j\in\mN$. We thus have $\sB=\{T_j\,|\, j\in\mN\}$ by Lemma~\ref{LemLem}. That $\Ob$ is an isomorphism thus follows from Lemma~\ref{Lemmuj} and Theorem~\ref{ThmSuff2}.
The result then follows easily from Proposition~\ref{PropConc}.
\end{proof}

\subsection{Higher rank}

\subsubsection{}For reductive groups $G$ of higher rank we will have more (thick) tensor ideals than the ones coming from the Frobenius kernels. For instance, in \cite[Proposition~12]{AndersenProc} it is showed that the maximal tensor ideal corresponds to
$$\{T(\lambda)\,|\,\lambda\in X^+ \mbox{ with } \langle \lambda,\alpha^\vee\rangle \ge p-1 \,,\, \mbox{for some~$\alpha\in R^+$}\}.$$
For~$G=\GL_n$ with $n>2$, this ideal thus strictly contains all proper ideals of Subsection~\ref{SecFK}.

\subsubsection{}\label{SecObstGLm}
In general, we will also no longer have a one-to-one correspondence between tensor ideals in $\Ti(G)$ and thick ideals in the Grothendieck ring $[\Ti(G)]_{\oplus}$. By Corollary~\ref{CorNoIso}, it suffices to find $\lambda\in X^+$ such that $\dim\Hom_G(\unit, T(\lambda))>1$. By equation~\eqref{standardstuff}, we thus need $\lambda\in X^+$ with~$(T(\lambda):\nabla(0))>1$. When $n>2$, these are known to exist for~$\GL_n$. In fact, already for~$G=\SL_3$, the value $(T(\lambda):\nabla(0))$ is actually conjectured to grow exponentially with~$\lambda$, for appropriate $\lambda\in X^+$, see~\cite{LW}.
\begin{prop}\label{ObstGLm}
The decategorification map $\Ob: \TId(\Ti(\GL_n))\tto \Ide([\Ti(\GL_n)]_{\oplus})$ is not an isomorphism when $ n>2$.
\end{prop}


\part{Deligne categories}

Fix an algebraically closed field~$\mk$ of characteristic zero. We will study three monoidal categories~$\RG$, $\RO$ and $\RP$ which satisfy conditions (I)-(V).

\section{Deligne categories and Brauer algebras}\label{SecBoring}

In this section, we use the explicit decomposition multiplicities of cellular diagram algebras in~\cite{BrMult, PB2, Martin} to show that for Deligne categories we can construct an auxiliary monoidal category $\cS$ satisfying the assumptions of~\ref{tuple}.

\subsection{The orthogonal case}
In~\cite[Section~9]{Deligne}, for any unital commutative ring $K$ and~$t\in K$, the monoidal category~${\rm Rep}(O(t),K)$ is introduced, see also~\cite[Section~2]{ComesHei}. 

\subsubsection{}For~$\delta\in\mk$, we write~$\RO={\rm Rep}(O(\delta),\mk)$, and this category is of the form
$$\RO\;:=\; \left(\ROz\right)^{\oplus\sharp},$$
for a strict monoidal category~$\ROz$, known as the {\em Brauer category}, see \cite{BrCat}. We write~$\cC=\RO$ and~$\cC_0=\ROz$.

Then $\cC_0$ is a skeletal monoidal category with~$\Ob\cC_0=\mN$ and~$i\otimes j=i+j$, so in particular $\unit_{\cC_0}=0$. The morphism space $\cC_0(i,j)$ consists of the $\mk$-linear combinations of~$(i,j)$-Brauer diagrams. The latter are similar to the Temperley-Lieb diagrams of \ref{TLdiag}, except that now crossings are allowed. For any such Brauer diagram~$d\in\cC_0(i,j)$, we denote by~$d^\ast\in\cC_0(j,i)$ the diagram obtained by reflection with respect to a horizontal axis. The composition of diagrams is given by concatenation of diagrams with evaluation of loops at $\delta\in\mk$. In particular, we have 
$$\cC_0(0,0)=\mk,\qquad\mbox{and}\qquad B_r(\delta)\;:=\;\cC_0(r,r)$$
is the Brauer algebra for~$r\in\mN$. This algebra is cellular, by \cite[Section~4]{CellAlg}. The cell modules are given by~$W_r(\lambda)$, for~$\lambda$ a partition of an element in~$\{r-2i\,|\, 0\le i\le r/2\}$. 
The simple modules are labelled by the same set of partitions, excluding $\varnothing$ when $\delta=0$, see e.g. \cite{CellAlg, Martin, BrMult, Borelic}. We write~$L_r(\lambda)$ for the corresponding simple module.

By \cite[Theorem~3.5]{ComesHei}, we have a bijection~$\Par\stackrel{\sim}{\to}\inde\cC$, given by~$\lambda\mapsto R(\lambda)$. We can take $R(\lambda)=(r,e_{\lambda})$, where $r=|\lambda|\in\Ob\cC_0$ and~$e_\lambda$ is a primitive idempotent in~$B_r(\delta)=\cC_0(r,r)$ corresponding to~$L_r(\lambda)$.
From \cite[Section~3]{ComesHei}, for~$\lambda\in\Par$ and~$k\in\mN=\Ob\cC_0$, we have
\begin{equation}\label{Rvsn}R(\lambda)\;\inplus\;k\quad\mbox{if and only if \hspace{3mm}$k-|\lambda|\in 2\mN$.}\end{equation} We have $\unit=R(\varnothing)$, which corresponds to~$0\in \Ob\cC_0$. Hence, $1\not=\unit_{\cC}=0\not=\zero_{\cC}$.

\subsubsection{}It is easy to see that~$\cC_0$ is rigid, with~$i^\vee=i$, for all~$i\in\mN$. For instance, for~$\ev_i$ we can take $(2i,0)$-Brauer diagrams of the form
\begin{equation}\label{eqDual}\begin{tikzpicture}[scale=0.9,thick,>=angle 90]

\draw (0,0) to [out=50,in=-180] +(1.8,.8);
\draw (1.8,.8) to [out=0,in=130] +(1.8,-.8);
\draw (0.6,0) to [out=50,in=-180] +(1.8,.8);
\draw (2.4,.8) to [out=0,in=130] +(1.8,-.8);
\draw (1.2,0) to [out=50,in=-180] +(1.8,.8);
\draw (3,.8) to [out=0,in=130] +(1.8,-.8);
\draw (1.8,0) to [out=50,in=-180] +(1.8,.8);
\draw (3.6,.8) to [out=0,in=130] +(1.8,-.8);
\draw (2.4,0) to [out=50,in=-180] +(1.8,.8);
\draw (4.2,.8) to [out=0,in=130] +(1.8,-.8);

\end{tikzpicture}
\end{equation}
and correspondingly we take~$\co_i=(\ev_i)^\ast$.
Note that we have composed the canonical construction of~$\ev_i$ from $\ev_1=\cap$, through iterative use of~\eqref{eqXYd}, with an isomorphism of~$i$.
It follows that~$\ev_i\circ (d\otimes I^{\otimes i})=\ev_i\circ(I^{\otimes i}\otimes d^\ast)$, for any diagram~$d\in B_i(\delta)$, with $I$ the identity morphism of $1\in\Ob\cC_0$.  Since $\ast:B_r(\delta)\to B_r(\delta)$ is the involution in the cell datum of~$B_r(\delta)$, see \cite[Theorem~4.10]{CellAlg}, it preserves simple modules. It then follows in particular that~$\cC$ is rigid, with~$R(\lambda)^\vee\simeq R(\lambda)$, for all~$\lambda\in\Par$.

\subsubsection{}\label{DefmnO}When $\delta\not\in\mZ$, the category $\cC=\RO$ is semisimple, see \cite[Th\'eor\`eme~9.7]{Deligne}. Theorem~\ref{Thm1} then demonstrates in particular that~$\cC$ does not admit any non-trivial tensor ideals. We therefore {\em henceforth restrict to the case $\delta\in\mZ\subset\mk$.} For~$j\in\mZ_{>0}$, we set
$$\mm_j=\begin{cases}\delta+2j-2&\mbox{if $\delta>0$,}\\
2j-2&\mbox{if $\delta\in -2\mN$,}\\
2j-1&\mbox{if $\delta\in -2\mN-1$,}\end{cases}\quad\mn_j=\frac{\mm_j-\delta}{2}=\begin{cases}j-1&\mbox{if $\delta>0$,}\\
j-\delta/2-1&\mbox{if $\delta\in -2\mN$,}\\
j-\delta/2-\frac{1}{2}&\mbox{if $\delta\in -2\mN-1$,}\end{cases}$$
and $\mr_j:=(\mm_j+1)(\mn_j+1)$.

\begin{lemma}\label{Lem1O}
Consider the set~$\Upsilon=\{\nu^{(j)}\,|\,j\in\mN\}\subset\Par$, given by 
$$\nu^{(0)}=\varnothing\qquad \mbox{and}\qquad \nu^{(j)}=((2\mn_j+2)^{\mm_j+1}),\,\mbox{ for~$j>0$.}$$
We have $\sB=\{R(\nu)\,|\, \nu\in \Upsilon \}$. Moreover, for all~$\lambda\in\Par$, we have
$$\dim_{\mk}\cC(\unit,R(\lambda))\;=\;\begin{cases} 1&\mbox{if $\lambda\in\Upsilon$,}\\
0&\mbox{otherwise.}
\end{cases}$$
\end{lemma}
\begin{proof}
Take $\lambda\in\Par$ and set~$r=|\lambda|$. By definition of~$R(\lambda)$ we have
$$\dim\cC(\unit, R(\lambda))\;=\;\dim_{\mk}e_\lambda\cC_0(0,r)\;=\;[\cC_0(0,r):L_r(\lambda)],$$
where we interpret~$\cC_0(0,r)$ as a left $B_r(\delta)$-module.
By construction, $\cC_0(0,r)=0$ if $r$ is odd. Assume thus that~$r$ is even.
By \cite[Example~8.6.5]{Borelic}, the left $B_r(\delta)$-module~$\cC_0(0,r)$ is precisely the cell module~$W_r(\varnothing)$. The multiplicities~$[W_r(\varnothing):L_r(\lambda)]$ have been calculated in~\cite{Martin, BrMult}.

We follow the approach of~\cite[Section~5]{BrMult} and assume familiarity of the reader with the combinatorics defined {\it loc. cit.} Assume $\delta=2s$, for~$s\in \mZ_{>0}$, the other cases follow similarly. To any partition $\lambda$ we associate the strictly decreasing sequence $x$ of integers
$$(x_1,x_2,x_3,\cdots):=(\lambda_1^t-s,\lambda^t_2-s-1,\lambda_3^t-s-2,\cdots).$$
Then we draw, on an invisible real axis, at each $n\in\mN\subset\mR$ a symbol $\circ$ if neither $\pm n$ appears in $x$, a $\wedge$ if $n$ appears and a $\vee$ if $-n$ appears, resulting in $\times$ if both $n$ and $-n$ appear.
The diagram associated to~$\varnothing$ is then given by
$$\circ\;\circ\;\cdots\;\circ\;\vee\;\vee\;\vee\;\cdots,$$
where we have $s$ times~$\circ$ and infinitely many~$\vee$. By \cite[Theorem~5.8]{BrMult}, the multiplicity $[W_r(\varnothing):L_r(\lambda)]$ is $1$ if the diagram of~$\lambda$ can be obtained from the above one by changing an even number of~$\vee$ to~$\wedge$ and if the resulting curl diagram for~$\lambda$ is ``oriented'' with respect to the above diagram, and $0$ otherwise. The rules for orientation mean that~$[W_r(\varnothing):L_r(\lambda)]\not=0$ implies the curl diagram of~$\lambda$ cannot contain any caps, which means that the diagram of~$\lambda$ must be of the form
$$\circ\;\circ\;\cdots\;\circ\;\wedge\;\wedge\;\cdots\;\wedge\;\wedge\;\vee\;\vee\;\vee\;\cdots,$$
with~$s$ times~$\circ$, an even number (say $2j$) times~$\wedge$ and infinitely many~$\vee$. It follows easily that the resulting curl diagram is oriented and that the above diagram corresponds to $\nu^{(j)}$.
\end{proof}

\subsubsection{}\label{defSCO}
Consider the $\mk$-linear category $\cS_0$ with $\Ob\cS_0=\mN$ which has as endomorphism algebra of $i\in\mN$ the group algebra $\mk \SG_i$. This is a monoidal category,  with $i\otimes j=i+j$, where the tensor product of morphisms corresponds to the canonical group monomorphism $\SG_i\times \SG_j\to \SG_{i+j}$. We set $\cS=(\cS_0)^{\oplus \sharp}$. Since $\charr(\mk)=0$, we can alternatively define $\cS$ as 
$$\cS\;=\; \bigoplus_{i\in\mN}\mk\SG_i\mbox{-mod},$$
where the tensor product is now given by the induction product mentioned in~\ref{SecSG}. We have $\Par=\inde\cS$, where we denote Specht modules by their partition, and $\ell :\Par\to\mN$, $\lambda\mapsto |\lambda|$ makes~$\cS$ an $l$-controlled monoidal category as in Definition~\ref{DefContr}. The partial order $\preceq$ is in this case given by $\lambda\preceq \mu$ if and only if $\lambda\subseteq\mu$ (inclusion of partitions). We have a monoidal functor~$\Tmod_0:\cS_0\to\cC_0$, since we can interpret $\cS_0$ as the subcategory of $\cC_0$ containing all objects, but only morphisms in the span of diagrams containing exclusively propagating lines. This extends to a monoidal functor~$\Tmod:\cS\to\cC$.

\begin{prop}\label{TupleO}
The tuple $(\cC,\Tmod,\cS,\ell)$ in~\ref{defSCO} satisfies the conditions in~\ref{tuple}
\end{prop}
\begin{proof}
We have already established that~(b) is satisfied, with $\phi=\id_{\Par}$. Condition~(c) follows from Lemma~\ref{Lem1O}. We check condition~(a). For~$\lambda\vdash r$, take a primitive idempotent $e^0_\lambda$ in $\mk\SG_r$ corresponding to the simple Specht module~$\lambda$. Then $e^0_\lambda$ decomposes inside $B_r(\delta)$ as a sum of mutually orthogonal primitive idempotents. To describe which ones appear we observe that
$$\dim \Hom_{B_r(\delta)}(B_r(\delta)e^0_\lambda, L_r(\mu))\;\le\;\dim \Hom_{\mk\SG_r}(\mk\SG_re^0_\lambda, \res W_r(\mu))\;=\;[\res W_r(\mu):\lambda],$$
for $\res$ the restriction functor from $B_r(\delta)$-modules to $\mk\SG_r$-modules.
By construction of $W_r(\mu)$, see e.g. \cite[Section~7]{Borelic}, the module~$\res W_r(\mu)$ is a direct summand of $\mu\boxtimes (2)^{\boxtimes \frac{|\lambda|-|\mu|}{2}}$. Hence, condition~(a) follows from the Littlewood-Richardson rule and $\Tmod(\lambda)=(r,e_\lambda^0)$.
Conditions (d), (e) and (f) follow from Lemma~\ref{LemLambdanuO} below and Lemma~\ref{LemRect}(i).
\end{proof}

\begin{lemma}
\label{LemLambdanuO}
With notation as in~\ref{tuple}, for the tuple in~\ref{defSCO} and $j\in\mZ_{>0}$ we have
$$\Lambda_j=\{\lambda\vdash \mr_j\,|\, \mbox{$\lambda$ is $(\mm_j+1)\times (2\mn_j+2)$-self-dual}\}.$$
For~$\lambda\in \Lambda_j$, we have $\dim\cC(R(\lambda),R(\lambda))=1$.
\end{lemma}
\begin{proof}
The description of $\Lambda_j$ follows from Lemma~\ref{LemRect}(ii). Assume now that $\delta=2s>0$. Then the elements of $\Lambda_j$ correspond to the $x\in\mZ^{\oplus\mN}$ (associated as in the proof of Lemma~\ref{Lem1O}) for which $x_i+x_{2j+1-i}=0$. It follows from \cite[Theorem 5.1]{BrMult} that the simple $B_{\mr_j}(\delta)$-module~$L_{\mr_j}(\lambda)$ is projective for~$\lambda\in\Lambda_j$. The same property can be proved similarly for the other values of $\delta$. This implies
$$\dim\cC(R(\lambda),R(\lambda))=\dim e_\lambda B_{\mr_j}(\delta) e_\lambda=1,$$
which concludes the proof.
\end{proof}

\subsection{The general linear case}
In~\cite[Example~1.27]{DM} or \cite[Section~10]{Deligne}, for any unital commutative ring $K$ and~$t\in K$, the monoidal category~${\rm Rep}(GL(t),K)$ is introduced.
\subsubsection{}For~$\delta\in\mk$, we write~$\RG={\rm Rep}(GL(\delta),\mk)$, and this category is of the form
$$\RG\;:=\; \left(\RGz\right)^{\oplus\sharp},$$
for a strict monoidal category~$\RGz$ known as the {\em oriented (or walled) Brauer category}, see~\cite[Section~3]{ComesWil} or~\cite[Section~4]{ES}. We write~$\cC=\RG$ and~$\cC_0=\RGz$.

The set $\Ob\cC_0$ consists of the finite sequences of symbols $\bullet$ and~$\circ$ (or~$\uparrow$ and~$\downarrow$ in e.g. \cite[\S 1.9]{Brundan}), in particular $\unit$ is the empty sequence. Switching two symbols in the sequence actually yields an isomorphic object. We will therefore mainly work with the objects~$[k,l]$ representing $k$ times~$\bullet$ followed by~$l$ times~$\circ$. The morphisms space $\cC_0([k,l],[k',l'])$ is the $\mk$-linear span of~$(k+l,k'+l')$-Brauer diagrams such that propagating lines connect dots of the same colour and cups and caps connect dots of different colour.
We have
$$\cC_0(\unit,\unit)\simeq\mk,\quad\mbox{and}\quad B_{k,l}(\delta)\;=\;\cC_0([k,l],[k,l])$$ 
is the walled Brauer algebra, for~$k,l\in\mN$. The simple~$B_{k,l}(\delta)$-modules are labelled as $L_{k,l}(\lambda^{\bullet},\lambda^{\circ})$ by pairs~$[\lambda^{\bullet},\lambda^{\circ}]\in\Par\times\Par$, with~$\lambda^{\bullet}\vdash k-t$ and~$\lambda^{\circ}\vdash l-t$, for some~$t\le \min(k,l)$, where we exclude $[\varnothing, \varnothing]$ when $\delta=0$, see e.g. \cite{BrMult, Borelic}. The cellularity of $B_{kl}(\delta)$ is well-known, see e.g.~\cite{Borelic}. We denote the cell modules by $W_{k,l}(\lambda^\bullet,\lambda^\circ)$.

By \cite[Section~4]{ComesWil}, we have a bijection~$\Par\times \Par\stackrel{\sim}{\to}\inde\cC$, given by~$[\lambda^\bullet,\lambda^\circ]\mapsto R(\lambda^\bullet,\lambda^\circ)$. We have $\unit=R(\varnothing,\varnothing)$, which corresponds to the empty sequence in~$ \Ob\cC_0$.  Furthermore, from \cite[Section~4]{ComesWil}, for~$[\lambda^\bullet,\lambda^\circ]\in\Par\times\Par$ and $k,l\in\mN$, we have
\begin{equation}\label{RvsnG}R(\lambda^\bullet,\lambda^\circ)\;\inplus\; [k,l]\quad\mbox{if and only if \hspace{3mm}$k-|\lambda^\bullet|=l-|\lambda^\circ|\in\mN$.}\end{equation}

\subsubsection{} As proved in~\cite[Section~3.3]{ComesWil}, the monoidal category~$\cC_0$ is rigid and $[k,l]^\vee\simeq [l,k]$, for~$[k,l]\in\Ob\cC_0\subset\Ob\cC$. For us it will be slightly more convenient to identify the dual of~$[k,l]$ with~$k$ times~$\circ$ followed by~$l$ times~$\bullet$. The diagram for the evaluation can then again be taken as in~\eqref{eqDual}. It follows again that for any diagram~$d\in B_{k,l}(\delta)$, we have $\ev_{[k,l]}\circ (d\otimes I^{\otimes(k+l)})=\ev_{[k,l]}\circ (I^{\otimes (k+l)}\otimes d^\ast)$.
It follows, as for~$\RO$, that~$\cC$ is rigid and that
\begin{equation}\label{rigidG}R(\lambda^\bullet,\lambda^{\circ})^\vee\;\simeq\; R(\lambda^\circ,\lambda^\bullet),\qquad\mbox{for all~$\lambda^\bullet,\lambda^\circ\in\Par$}.\end{equation}

\subsubsection{}\label{DefmnG}When $\delta\not\in\mZ$, we have that~$\cC=\RG$ is semisimple abelian, see \cite[Th\'eor\`eme~10.5]{Deligne} or \cite[Theorem~4.8.1]{ComesWil}. We therefore henceforth restrict to the case $\delta\in\mZ\subset\mk$. For~$j\in\mZ_{>0}$, we set
$$\mm_j=\begin{cases}\delta+j-1&\mbox{if $\delta\ge0$,}\\
j-1&\mbox{if $\delta\le 0$,}\\
\end{cases}\quad\mbox{and}\quad\mn_j=\mm_j-\delta=\begin{cases}j-1&\mbox{if $\delta\ge0$,}\\
j-\delta-1&\mbox{if $\delta\le 0$}\end{cases}$$
and $\mr_j:=(\mm_j+1)(\mn_j+1).$

\begin{lemma}\label{Lem1G}
Consider the set~$\Upsilon=\{[\nu^{(j)\bullet},\nu^{(j)\circ}]\,|\,j\in\mN\}\subset\Par\times\Par$, given by 
$$\nu^{(0)\bullet}=\nu^{(0)\circ}=\varnothing\quad\mbox{and}\quad \nu^{(j)\bullet}=\nu^{(j)\circ}=((\mn_j+1)^{\mm_j+1}), \;\mbox{ for~$j>0$.}$$
We have $\sB=\{R(\nu^\bullet,\nu^\circ)\,|\, [\nu^\bullet,\nu^\circ]\in \Upsilon \}$. Moreover, for any~$[\lambda^\bullet,\lambda^\circ]\in\Par\times\Par$, we have
$$\dim_{\mk}\cC(\unit,R(\lambda^\bullet,\lambda^\circ))\;=\;\begin{cases} 1&\mbox{if $[\lambda^\bullet,\lambda^{\circ}]\in\Upsilon$,}\\
0&\mbox{otherwise.}
\end{cases}$$
\end{lemma}
\begin{proof}
As in the proof of Lemma~\ref{Lem1O}, it follows that
$$\dim\cC(\unit,R(\lambda^\bullet,\lambda^\circ))\;=\;[W_{r,r}(\varnothing,\varnothing):L_{r,r}(\lambda^\bullet,\lambda^\circ)],$$
with~$r=\max\{|\lambda^\bullet|,|\lambda^\circ|\}$, if $|\lambda^\bullet|+|\lambda^{\circ}|$ is even. Multiplicities for~$|\lambda^\bullet|+|\lambda^{\circ}|$ odd (more generally when $|\lambda^\bullet|\not=|\lambda^{\circ}|$) are zero. The above multiplicities are determined by \cite[Theorem~4.10]{BrMult}.
\end{proof}

\subsubsection{}\label{defSCG} We define by $\cS_0$ the subcategory of $\cC_0$ with same objects, but with morphism spaces spanned by diagrams containing only propagating lines. This yields a monoidal functor~$\Tmod:\cS\to\cC$ for~$\cS=\cS_0^{\oplus \sharp}$. We have an equivalence of monoidal categories
$$\cS\;\simeq\; \bigoplus_{i,j\in\mN}(\mk\SG_i\otimes\mk\SG_j)\mbox{-mod},$$
where the monoidal structure on the right is derived from the induction product. In particular, we have $\inde\cS=\Par\times\Par$, where a pair $[\lambda^\bullet,\lambda^\circ]$ of partitions corresponds to the exterior tensor product of two Specht modules $\lambda^\bullet\boxtimes\lambda^{\circ}$. Hence $\cS$ is an $l$-controlled category for~$\ell:\Par\times\Par\to \mN$ defined as $\ell([\lambda^\bullet,\lambda^\circ])=|\lambda^\bullet|+|\lambda^\circ|$. Furthermore, we have $(\lambda^\bullet,\lambda^\circ)\preceq(\mu^\bullet,\mu^\circ)$ if and only if $\lambda^\bullet\subset\mu^\bullet$ and $\lambda^\circ\subset\mu^\circ$.

\begin{prop}\label{TupleG}
The tuple $(\cC,\Tmod,\cS,\ell)$ in~\ref{defSCG} satisfies the conditions in~\ref{tuple}
\end{prop}
\begin{proof}
Mutatis mutandis the proof of Lemma~\ref{TupleO}, using Lemmata~\ref{Lem1G}, \ref{LemLambdanuG} and equation~\eqref{rigidG}.
\end{proof}

\begin{lemma}
\label{LemLambdanuG}
For the tuple in~\ref{defSCG}, with notation as in~\ref{tuple}, we have for~$j\in\mZ_{>0}$
$$\Lambda_j=\{[\lambda^{\bullet},\lambda^{\circ}]\in\Par\times\Par\,|\, \mbox{$\lambda^\bullet$ and $\lambda^{\circ}$ are $(\mm_j+1)\times (\mn_j+1)$-dual}\}.$$
For every $[\lambda^\bullet,\lambda^\circ]\in\Lambda_j$, we have $\dim\cC(R(\lambda^\bullet,\lambda^\circ),R(\lambda^\bullet,\lambda^\circ))=1$.
\end{lemma}
\begin{proof}
The description of $\Lambda_j$ is immediate from Lemma~\ref{LemRect}(ii). The observation on the dimension follows from \cite[Theorem~4.1]{BrMult} as in the proof of Lemma~\ref{LemLambdanuO}.
\end{proof}

\subsection{The periplectic case}
\subsubsection{}In~\cite{Kujawa}, an analogue of~$\ROz$ was introduced, the {\em periplectic Brauer category~$\RPz$}, see also \cite{Vera}. We set $\cC_0:=\RPz$. We have again~$\Ob\cC_0=\mN$ and~$\cC_0(i,j)$ is again the $\mk$-span of~$(i,j)$-Brauer diagrams. The composition of morphisms now corresponds to concatenation of diagrams up to a possible minus sign, and loops are evaluated at zero. For the correct sign rules we refer to \cite{Kujawa}, where we use the convention, as in~\cite{PB1}, that unmarked Brauer diagrams are to be interpreted as diagrams with standard marking in~\cite{Kujawa}. The category~$\cC_0:=\RPz$ is a monoidal {\em super}category, see Appendix~\ref{SecSuper}, with~$i\otimes j=i+j$. The $\mF_2$-grading is such that~$\cC_0(i,j)$ is homogeneous of the same partity as $(i-j)/2$. 
The periplectic Brauer algebras
$A_r:=\cC_0(r,r)$ were first introduced in~\cite{Moon}. Note that the inherited $\mF_2$-grading of~$A_r$ is reduced, meaning that~{\em$A_r$ is purely even}.
The simple modules~$L_r(\lambda)$ of~$A_r$ are labelled by all partitions $\lambda$ of elements in~$\{r-2i\,|0\le i<r/2\}$, see \cite{Kujawa, PB1}. 

We set~$\RP:=\cC_0^{\oplus\sharp}$  and by \cite[Section~2.2]{PB3} we have a bijection~$\Par\stackrel{\sim}{\to}\inde\cC$, given by~$\lambda\mapsto R(\lambda)$. We have $\unit=R(\varnothing)$, which corresponds to~$0\in \Ob\cC_0$.
As follows from \cite[Section~2.2]{PB3}, for~$\lambda\in\Par$ and~$k\in\mN=\Ob\cC_0$, we have
\begin{equation}\label{RvsnP}R(\lambda)\;\inplus\; k\quad\mbox{if and only if \hspace{3mm}$k-|\lambda|\in 2\mN$.}\end{equation}  

\subsubsection{}\label{SecDualP} Define $\ev_i\in \cC_0(2i,0)$ as in equation~\eqref{eqDual} and~$\co_i=(\ev_i)^\ast$. Application of the diagram calculus of \cite{Kujawa} shows that 
$$(\ev_i\otimes I^{\otimes i})\circ (I^{\otimes i}\otimes \co_i)=I^{\otimes i}\qquad\mbox{and}\qquad ( I^{\otimes i}\otimes \ev_i)\circ ( \co_i \otimes I^{\otimes i})=(-1)^i I^{\otimes i}.$$
It follows that~$\cC_0$ is rigid, in the sense of Appendix \ref{SecSuper}, with~$i^\vee=i$, for all~$i\in\mN$ and~$i$ has a dual of parity $i$ mod $2$. For~$d$ an $(i,i)$-Brauer diagram, it follows that~$\ev_i\circ (d\otimes I^{\otimes})=\ev_i\circ(I^{\otimes i}\otimes \varphi(d))$, for~$\varphi$ the anti-automorphism of~$A_r$ of \cite[Remark~4.1.4]{PB1} which exchanges the simple modules~$L_r(\lambda)$ and~$L_r(\lambda^t)$. Consequently, $\cC$ is rigid, with 
$$R(\lambda)^\vee\;\simeq\;R(\lambda^t),\qquad\mbox{for all~$\lambda\in \Par$,}$$
and~$R(\lambda)$ has a dual of parity $|\lambda|$ mod $2$.

\begin{lemma}\label{Lem1P}
Consider the set~$\Upsilon=\{\nu^{(j)}\,|\,j\in\mN\}\subset\Par$, given by~$\nu^{(j)}=(j+1)^j$. We have $\sB=\{R(\nu)\,|\, \nu\in \Upsilon \}$. Moreover, we have
$$\dim_{\mk}\cC(\unit,R(\lambda))\;=\;\begin{cases}
1&\mbox{if}\;\;\lambda\in\Upsilon\\
0&\mbox{otherwise.}
\end{cases}$$
\end{lemma}
\begin{proof}
For~$r$ odd we have $\cC_0(0,r)=0$, and for~$r$ even \cite[Lemma~4.4.1]{PB1} and \cite[equation~(4.6)]{PB1} imply an isomorphism of~$A_r$-modules
$\cC_0(0,r)\simeq W_r(\varnothing),$
where $W_r(\varnothing)$ is the cell module over~$A_r$ introduced in~\cite[Section~4.6]{PB1}. As in the proof of Lemma~\ref{Lem1O}, it thus follows that
$$\dim\cC(\unit, R(\lambda))\;=\;[W_r(\varnothing):L_r(\lambda)],\qquad\mbox{ for all~$\lambda\vdash r$}.$$
 The claim then follows from \cite[Theorem~1]{PB2}.
\end{proof}

\subsubsection{}\label{defSCP}
We clearly have a monoidal (super)functor~$\Tmod:\cS\to\cC$, for~$\cS$ the monoidal category in~\ref{defSCO} which can also be interpreted as a subcategory of $\RP$.

\begin{prop}\label{TupleP}
The tuple $(\cC,\Tmod,\cS,\ell)$ in~\ref{defSCO} satisfies the conditions in~\ref{tuple}
\end{prop}
\begin{proof}
Condition~(a) follows from the proof of \cite[Corollary~6.2.7]{PB1}. We established in~\ref{SecDualP} that~(b) is satisfied with $\phi:\Par\to\Par$ given by transposition.
Condition~(c) follows from Lemma~\ref{Lem1P}. Conditions (d), (e) and (f) follow from Lemma~\ref{LemLambdanuO} below and Lemma~\ref{LemRect}(i).
\end{proof}

\begin{lemma}
\label{LemLambdanuP}
With notation as in~\ref{tuple}, for the tuple in~\ref{defSCO} and $j\in\mZ_{>0}$ we have
$$\Lambda_j=\{\lambda \,|\, \mbox{$\lambda^t$ and $\lambda$ are $j\times(j+1)$-dual}\}=\{(j,j-1,\ldots,1)\}.$$
For~$\lambda\in \Lambda_j$, we have $\dim\cC(R(\lambda),R(\lambda))=1$.
\end{lemma}
\begin{proof}
The description of $\Lambda_j$ follows from Lemma~\ref{LemRect}(ii). All morphism spaces between two indecomposable objects in $\cC$ are one-dimensional, by \cite[Proposition~A.3.1]{PB2}.\end{proof}


\section{Tensor ideals in Deligne categories}\label{SecClass}

\subsection{The orthogonal case}\label{SecOideal} Fix $\delta\in\mZ\subset\mk$.
By~\cite[Theorem~5.6]{LZ-FFT}, for every~$(m,n)\in\mN\times\mN$ with~$\delta=m-2n$, we have a {\em full} monoidal functor
$$\Fmod_{m,n}:\;\RO\,\to\, \Rep_{\mk}\OSp(m|2n),$$
determined by the property that it maps~$R(\Box)$ to the natural representation $\mk^{m|2n}$.
This is the {\em first fundamental theorem of invariant theory},
see also~\cite[Theorem~3.4]{Vera} or \cite[Section~3.13]{DLZ}. Here, $\Rep_{\mk}\OSp(m|2n)$ is the category of algebraic finite dimensional representations of the algebraic supergroup $\OSp(m|2n)$. We refer to \cite{ES, ComesHei} for details on that category.

\begin{thm}\label{ThmO}
The tensor ideals in~$\cC=\RO$ form a set~$\{\cJ_i\,|\,i\in\mN\}$ with
$$\cC\;=\;\cJ_0\;\supsetneq\;\cJ_1\;\supsetneq\cJ_2\;\supsetneq\; \cJ_3\;\supsetneq\; \cdots$$
and $\Ob:\TId(\cC)\tto \Ide([\cC]_{\oplus})$ is an isomorphism. For~$j\in\mZ_{>0}$, we have the following descriptions of~$\cJ_j$, with $\mm_j,\mn_j$ as in~\ref{DefmnO}.
\begin{enumerate}[(i)]

\item For~$X,Y\in\Ob\cC$, the $\mk$-module~$\cJ_j(X,Y)$ consists of all morphisms which factor as $X\to Z\to Y$, with~$Z$ a direct sum of objects $R(\lambda)$, with~$\lambda\in\Par$ satisfying
$$\lambda_{i}^t+\lambda^t_{2\mn_j+3-i}> \mm_j,\qquad\mbox{for all $1\le i\le \mn_j+1$.}$$
\item We have $\cJ_j=\ker \Fmod_{\mm_j,\mn_j}$.
\item The submodule~$\Mmod_j:=\Psi(\cJ_j)=\cJ_{j}(\unit,-)\in\Sub(\Pmod_{\unit})$ is determined by
$$\Mmod_j(R(\nu^{(k)}))=0\mbox{ if $k<j\quad$ and }\quad \Mmod_j(R(\nu^{(k)}))=\cC(\unit,R(\nu^{(k)}))\;\mbox{ if $k\ge j\;$},$$
with~$\nu^{(j)}$ as in Lemma~\ref{Lem1O}.
\end{enumerate}
\end{thm}
\begin{proof}
By Proposition~\ref{TupleO}, we can apply Theorem~\ref{ThmSC}. All statements except part (ii) then follow immediately from Theorem~\ref{ThmSC}, using Lemma~\ref{LemLambdanuO}.

To prove part (ii), we consider a commuting diagram of monoidal functors
$$\xymatrix{
\cS\ar[rr]^-{\Gmod_j}\ar[d]^{\Tmod}&& \Rep_{\mk}\GL(\mm_j|2\mn_j)\ar[d]\\
\cC\ar[rr]^-{\Fmod_{\mm_j,\mn_j}}&& \Rep_{\mk}\OSp(\mm_j|2\mn_j).
}$$
The right vertical arrow represents the restriction functor. The functor~$\Gmod_j$ corresponds to Schur-Weyl duality for the general linear supergroup, see \cite{Berele}, mapping $k\in\mN$ to $V^{\otimes k}$ with $V=\mk^{m|2n}$. In particular, the image of $\Gmod_j$ is contained in the semisimple category of polynomial representations. By \cite{Berele}, we have $\Gmod_j(\lambda)=\zero$ for~$\lambda\in\Par$ if and only if $\lambda_{\mm_j+1}>2\mn_j$. Consequently $\Gmod_j(\nu^{(j)})=\zero$, which implies~$\Fmod_{\mm_j,\mn_j}(R(\nu^{(j)}))=\zero$. Consequently, $\cJ_j\subset\ker\Fmod_{\mm_j,\mn_j}$. The above diagram also describes for which $\lambda\in\Par$ we have $\Fmod_{\mm_j,\mn_j}(\Tmod(\lambda))\not=\zero$ which shows that the ideals $\{\ker\Fmod_{\mm_j,\mn_j}\,|\,j\in\mN\}$ are two-by-two distinct. Part (ii) now follows.
\end{proof}

For each $j\in\mZ_{>0}$, we let $\cD_j$ be the full subcategory of~$\RO/\cJ_j$ with objects $I_{j-1}=\Ob(\cJ_{j-1})$. This is the natural realisation of the ``quotients'' of the filtration in Theorem~\ref{ThmO}. 

We let $\cD_j^{\op}$-mod be the full subcategory of $\cD_j^{\op}$-Mod of functors $\Mmod:\cD_j^{\op}\to\Ab$ such that $\Mmod(X)=0$ for all but finitely many $X\in \inde\cD_j$ and such that $\Mmod(X)$, which automatically gains the structure of a $\mk$-module, is finite dimensional, for all~$X\in\Ob\cC$.

\begin{cor}\label{CorO} Consider $(m,n)\in\mN\times\mN$ with~$\delta=m-2n$ and~$j\in\mZ_{>0}$ such that~$m=\mm_j$.
We have an equivalence of $\mk$-linear categories
$$\cD_j^{\op}\mbox{{\rm -mod}}\;\simeq\; \Rep_{\mk}\OSp(m|2n).$$
\end{cor}
\begin{proof}
By \cite[Lemma~7.5]{ComesHei}, every projective object in~$\Rep_{\mk}\OSp(m|2n)$ is in the image of~$\Fmod_{m|n}$. We denote the full subcategory of projective objects by~$\bP$. By \cite[Lemma~7.16]{ComesHei}, $\Fmod_{m,n}(X)$ is projective in~$\Rep_{\mk}\OSp(m|2n)$ if and only if $X\in\sI_{j-1}$. By Theorem~\ref{ThmO}, the functor~$\Fmod_{m,n}$ thus restricts to an equivalence $\cD_j\stackrel{\sim}{\to}\bP.$
It follows easily from the Yoneda lemma that we have an equivalence
$$\Rep_{\mk}\OSp(m|2n)\;\stackrel{\sim}{\to}\; \bP^{\op}\mbox{-mod},\quad M\mapsto \Hom_{\OSp(m|2n)}(-,M).$$
The combination of both equivalences concludes the proof.
\end{proof}
\begin{rem}
\begin{enumerate}[(i)]
\item Theorem~\ref{ThmO} yields in particular an alternative proof for \cite[Theorems~6.11, 7.3(ii) and 7.12 and Corollary~7.13]{ComesHei}.
\item Theorem~\ref{ThmO} for~$j=1$ and Lemma~\ref{PropAK} state that the kernels of the functors from Deligne categories to the module categories of O$(m)$, Sp$(2n)$ and $\OSp(1|2n)$ are given by the ideals of negligible morphisms, which was first proved in~\cite[Th\'eor\`eme~9.6]{Deligne}.
\item 
The combination of Theorem~\ref{ThmO}(i) and (ii) provides an affirmative answer to a question raised by Comes and Heidersdorf in~\cite[\S 8.1 (4)]{ComesHei}.
\item Corollary~\ref{CorO} yields, at least {\em in theory}, a means to describe~$\Rep_{\mk}\OSp(m|2n)$ diagrammatically. Diagrammatic realisations for this category have already been obtained by Ehrig and Stroppel in~\cite{ES}.
\end{enumerate}
\end{rem}

\begin{rem}
As already observed in \cite{Comes, ComesHei, PB3}, each tensor~$\Ob$-ideals in $\cC$, for~$\cC$ one of our Deligne categories, consists of the objects which are sent to zero by a monoidal functor~$\cC\to\Rep_{\mk}G$, for an affine algebraic supergroup scheme $G$. Since in the latter categories~$X\otimes Y\simeq \zero$ means either $X\simeq \zero$ or~$Y\simeq \zero$, this shows that all tensor~$\Ob$-ideals in Deligne categories are `prime' in the sense of \cite[Definition~2.1]{Balmer}.
\end{rem}

\subsection{The general linear case}\label{SecGideal}
Fix $\delta\in\mZ\subset\mk$.
For every~$(m,n)\in\mN\times\mN$ with~$\delta=m-n$, we have a {\em full} monoidal functor
$$\Fmod_{m,n}:\;\RG\,\to\, \Rep_{\mk}\GL(m|n),$$
determined by $[1,0]\mapsto \mk^{m|n}$.
This is the {\em first fundamental theorem of invariant theory},
see~\cite[Section~8.3]{ComesWil}, \cite[Theorem~3.2]{LZ-FFT}, and~\cite{Berele, Sergeev}.

\begin{thm}\label{ThmG}
The tensor ideals in~$\cC=\RG$ form a set~$\{\cJ_i\,|\,i\in\mN\}$ with
$$\cC\;=\;\cJ_0\;\supsetneq\;\cJ_1\;\supsetneq\cJ_2\;\supsetneq\; \cJ_3\;\supsetneq\; \cdots$$
and $\Ob:\TId(\cC)\tto \Ide([\cC]_{\oplus})$ is an isomorphism. For~$j\in\mZ_{>0}$, we have the following descriptions of~$\cJ_j$, with $\mm_j$ and $\mn_j$ as in~\ref{DefmnG}.
\begin{enumerate}[(i)]
\item For~$X,Y\in\Ob\cC$, the $\mk$-module~$\cJ_j(X,Y)$ consists of all morphisms which factor as $X\to Z\to Y$, with~$Z$ a direct sum of~$R(\lambda^\bullet,\lambda^\circ)$ with 
\begin{equation}\label{eqhookG}\lambda^\bullet_{l}+\lambda^{\circ}_{\mm_j-l+2}>\mn_j,\quad\mbox{ for all }\quad\;1\le l\le \mm_j+1.\end{equation}
\item We have $\cJ_j=\ker \Fmod_{\mm_j,\mn_j}$.
 \item The submodule~$\Mmod_j:=\Psi(\cJ_j)=\cJ_{j}(\unit,-)\in\Sub(\Pmod_{\unit})$ is determined by
$$\Mmod_j(R(\nu^{(k)\bullet},\nu^{(k)\circ}))=0\;\mbox{ if $k<j\;$ and }\; \Mmod_j(R(\nu^{(k)\bullet},\nu^{(k)\circ}))=\cC(\unit,R(\nu^{(k)\bullet},\nu^{(k)\circ}))\;\mbox{ if $k\ge j\;$}.$$
\end{enumerate}
\end{thm}
\begin{proof}
By Proposition~\ref{TupleG}, we can apply Theorem~\ref{ThmSC}. All statements except part (ii) then follow immediately from Theorem~\ref{ThmSC}, using Lemma~\ref{LemLambdanuG}.

To prove part (ii), we consider a commuting diagram of monoidal functors
$$\xymatrix{
\cS\ar[rr]^-{\Gmod_j}\ar[d]^{\Tmod}&& \Rep_{\mk}(\GL(\mm_j|\mn_j)\times\GL(\mm_j|\mn_j))\ar[d]\\
\cC\ar[rr]^-{\Fmod_{\mm_j,\mn_j}}&&\; \;\Rep_{\mk}\GL(\mm_j|\mn_j).
}$$
The right vertical arrow represents the restriction functor for the diagonal embedding. The functor~$\Gmod_j$ corresponds to Schur-Weyl duality for the general linear supergroup, see \cite{Berele}, mapping $[1,0]$ to $\mk^{\mm_j|\mn_j}$ interpreted as the natural representation for the first copy of $\GL(\mm_j|\mn_j)$ and $[0,1]$ to $(\mk^{\mm_j|\mn_j})^\ast$, the dual of the natural representation for the second copy of $\GL(\mm_j|\mn_j)$. Part (ii) now follows as in the proof of Theorem~\ref{ThmO}(ii).
\end{proof}

By \cite[Lemma~5.9]{He} (or \cite[Theorem~5.8]{He}), every projective object in~$\Rep_{\mk}\GL(m|n)$ is in the image of~$\Fmod_{m|n}$. By \cite[Theorem~5.12]{He} or (the proof of) \cite[Proposition~7.3]{He}, $\Fmod_{m,n}(X)$ is projective in~$\Rep_{\mk}\GL(m|n)$ if and only if $X\in\sI_{j-1}$, for~$j\in\mZ_{>0}$ such that~$m=\mm_j$. As in Corollary~\ref{CorO}, this allows us to conclude the following corollary for the category $\cD_j$ defined as the full subcategory of $\RG/\cJ_j$ with objects $\sI_{j-1}$.

\begin{cor}\label{CorG} Consider $(m,n)\in\mN\times\mN$ with~$\delta=m-n$ and~$j\in\mZ_{>0}$ such that~$m=\mm_j$.
 We have an equivalence of $\mk$-linear categories
$$\cD_j^{\op}\mbox{{\rm -mod}}\;\simeq\; \Rep_{\mk}\GL(m|n).$$
\end{cor}

\begin{rem}
\begin{enumerate}[(i)]
\item Theorem~\ref{ThmG} provides in particular alternative proofs for \cite[Theorem~8.7.1]{ComesWil}  and the main result of \cite{Comes}.
\item The result in Theorem~\ref{ThmG}(ii) agrees with the philosophy of \cite[Theorem~2]{EHS}. However, it seems that neither result implies the other directly. 
\item Corollary~\ref{CorG} yields, at least {\em in theory}, a means to describe~$\Rep_{\mk}\GL(m|n)$ diagrammatically. Diagrammatic realisations have already been obtained by Brundan and Stroppel in~\cite{BS}.
\end{enumerate}

\end{rem}

\subsection{The periplectic case}
By results in~\cite{DLZ, Kujawa}, we have a {\em full} monoidal superfunctor
$$\Fmod_n\,:\;\RP\,\to\, \Rep_{\mk} \Pe(n),\qquad\mbox{for all~$n\in\mN$,}$$
determined by $R(\Box)\mapsto \mk^{n|n}$, see~\cite[Theorem~5.2.1]{PB3}. Here, $\Pe(n)$ is the periplectic supergroup, see e.g.~\cite[Section~4.1]{DLZ}.

\begin{thm}\label{ThmP}
The tensor ideals in~$\cC=\RP$ form a set~$\{\cJ_i\,|\,i\in\mN\}$ with
$$\cC\;=\;\cJ_0\;\supsetneq\;\cJ_1\;\supsetneq\cJ_2\;\supsetneq\; \cJ_3\;\supsetneq\; \cdots$$
and $\Ob:\TId(\cC)\tto \Ide([\cC]_{\oplus})$ is an isomorphism. 
For~$j\in\mZ_{>0}$, we have the following descriptions of~$\cJ_j$.
\begin{enumerate}[(i)]
\item For~$X,Y\in\Ob\cC$, the $\mk$-module~$\cJ_j(X,Y)$ consists of all morphisms which factor as $X\to Z\to Y$, with~$Z$ a direct sum of objects $R(\lambda)$ with 
$\lambda_i\ge j+1-i$, for~$1\le i\le j$.
\item We have $\cJ_{j}=\ker \Fmod_{j-1}$.
 \item The submodule~$\Mmod_j:=\Psi(\cJ_j)=\cJ_{j}(\unit,-)\in\Sub(\Pmod_{\unit})$ is determined by
$$\Mmod_j(R(\nu^{(k)}))=0\;\mbox{ if $k<j\;$ and }\; \Mmod_j(R(\nu^{(k)}))=\cC(\unit,R(\nu^{(k)}))\;\mbox{ if $k\ge j\;$}.$$
\end{enumerate}
\end{thm}
\begin{proof}
By Proposition~\ref{TupleP}, we can apply Theorem~\ref{ThmSC}. All statements except part (ii) then follow immediately from Theorem~\ref{ThmSC}, using Lemma~\ref{LemLambdanuP}.

In~\cite[Theorem~5.2.1]{PB3}, it is proved that~$\Ob(\ker \Fmod_{n})=\sI_{n+1}$. This means the only tensor ideal in the classification which can be equal to~$\ker \Fmod_{n}$ is $\cJ_{n+1}$, which proves part~(ii). Alternatively we can copy the proof of Theorem~\ref{ThmO}(ii).
\end{proof}

By \cite[Lemma~8.3.2]{PB1}, all projective modules in $\Rep_{\mk}\Pe(n)$ are in $\im\Fmod_{n}$. By \cite[Theorem~5.3.1]{PB3} and the above theorem, we find the following corollary for the category $\cD_j$ defined as the full subcategory of $\RP/\cJ_j$ with objects $\sI_{j-1}$.

\begin{cor}\label{CorP} For~$n\in\mN$,
we have an equivalence of $\mk$-linear categories
$$\cD_n^{\op}\mbox{{\rm -mod}}\;\simeq\; \Rep_{\mk}\Pe(n).$$
\end{cor}

\begin{rem}
Corollary~\ref{CorP} yields, at least {\em in theory}, a means to describe $\Rep_{\mk}\Pe(n)$ diagrammatically. \end{rem}


\section{The second fundamental theorem of invariant theory}\label{SecSFT}
The previous section already yields the second fundamental theorem in a categorical version. In this section, we discuss the algebra version. We determine when we have an isomorphism between the relevant diagram algebra and the endomorphism algebra of a tensor power of the natural representation, and more generally describe the kernel of the morphism as an ideal.

\subsection{The orthosymplectic case}

\begin{thm}\label{SFTO} For an $\mF_2$-graded vector space~$V$ of dimension~$(m,2n)$ with~$\delta=m-2n$, we set~$r_c=(m+1)(n+1)$ and consider the surjective algebra morphism
$$\phi^r:\;B_r(\delta)\;\tto\;\End_{\OSp(V)}(V^{\otimes r}),\qquad\mbox{for~$r\in\mN$}.$$
\begin{enumerate}[(i)]
\item If~$r<r_c$, then $\phi^r$ is an isomorphism.
\item If~$r\ge r_c$, the kernel of~$\phi^r$ is generated as a two-sided ideal by a single element $F\otimes I^{r-r_c}$, with~$F\in B_{r_c}(\delta)$, such that~$F\circ d=0$, resp. $d\circ F=0$, for any Brauer diagram~$d$ which contains a cup, resp. cap.
\item The element $F$ can be chosen to be an idempotent if and only if $m\in\{0,1\}$ or~$n=0$. Under those assumptions, $F$ can simultaneously be chosen to be central in~$B_{r_c}(\delta)$. 
\end{enumerate}
\end{thm}

\begin{rem}${}$
\begin{enumerate}[(i)]
\item Theorem~\ref{SFTO}(i) was recently proved by Zhang in~\cite[Theorem~5.12]{Yang}, using a different approach, and was conjectured in~\cite{LZ-SFT}.
\item In~\cite{HX, LZ1, Yang}, Theorem~\ref{SFTO}(ii) is proved for the special cases  Sp$(2n)$, O$(m)$ and~$\OSp(1|2n)$. The results in~\cite{LZ1, BrCat, Yang} even provide explicit diagrammatic expressions for the generating element.
That in all other cases the ideals are still generated by one element is somewhat unexpected, see e.g.~\cite[Remark~5.9]{LZ-SFT}.
\item That the generating element for  O$(m)$ and Sp$(2n)$ can be chosen to be an idempotent as in Theorem~\ref{SFTO}(iii) is known by~ \cite{HX, LZ1}. That this is possible for~$\OSp(1|2n)$ is new.
\end{enumerate}
\end{rem}

We start the proof with the following proposition about $\cJ_j$ in $\cC=\RO$ in Theorem~\ref{ThmO}.

\begin{prop}\label{propDualO}
Consider $\lambda,\mu\in\Par$ and~$j\in\mN$ and recall~$\mr_j,\mm_j,\mn_j$ from \ref{DefmnO}.
\begin{enumerate}[(i)]
\item For~$a,b\in \mN=\Ob\cC_0$ with~$a+b<2\mr_j$, we have $\cJ_j(a,b)=0$. Hence we have 
$$\cJ_j(R(\lambda),R(\mu))=0,\quad\mbox{if $|\lambda|+|\mu|<2\mr_j$.}$$
\item If~$|\lambda|+|\mu|=2\mr_j$, we have
$$\dim_{\mk}\cJ_j(R(\lambda),R(\mu))\;=\;\begin{cases}1 &\mbox{if $\lambda$ and~$\mu$ are $(\mm_j+1)\times(2\mn_j+2)$-dual}\\
0&\mbox{otherwise}.
\end{cases}$$
\item 
\begin{enumerate}[(a)]
\item If~$|\lambda|=|\mu|=\mr_1$, we have $\cJ_1(R(\lambda),R(\mu))=0$, unless $\lambda=\mu$ and~$R(\lambda)\in \sI_1$.
\item If~$j>1$, there exist $\lambda,\mu\vdash \mr_j$ with~$\cJ_j(R(\lambda),R(\mu))\not=0$, $R(\lambda)\not\in \sI_j\not\ni R(\mu)$ and $\lambda\not=\mu$.
\end{enumerate}
\end{enumerate}
\end{prop}
\begin{proof}
We will freely use equation~\eqref{Rvsn}. Theorem~\ref{ThmO}(iii) implies that
$\cJ_j(0,i)=0$, for all~$\Ob\cC_0\ni i<2\mr_j=|\nu^{(j)}|$.
By Lemma~\ref{LemJJ}, we thus have $\cJ_{j}(a,b)\simeq\cJ_j(0,a+b)=0$, if $a+b<2\mr_j$.
This proves part~(i).

By Lemma~\ref{LemJJ}, we have $\cJ_j(R(\lambda),R(\mu))\simeq \cJ_j(\unit, R(\lambda)\otimes R(\mu))$. By Theorem~\ref{ThmO}(iii), we thus find that~$\dim_{\mk}\cJ_j(R(\lambda),R(\mu))$ is equal to the number of times that~$R(\nu^{(j)})$ occurs as a direct summand in~$R(\lambda)\otimes R(\mu)$. Part~(ii) thus follows from Lemmata~\ref{LemSC}(i) and~\ref{LemRect}.

Now we turn to part (iii). Assume that for~$\lambda,\mu\vdash \mr_1$, we have $\cJ_1(R(\lambda),R(\mu))\not=0$. It follows easily from part (ii) that this means that~$\lambda=\mu$. That $R(\lambda)\in\sI_1$ follows from Theorem~\ref{ThmO}.

For part~(iii)(b) it suffices, by part~(ii) and Theorem~\ref{ThmO}, to prove that there exist $\lambda,\mu\vdash \mr_j$ with $\lambda\not=\mu$ which are $(\mm_j+1)\times(2\mn_j+2)$-dual, for~$j>1$. This is a straightforward exercise.
\end{proof}

\begin{proof}[Proof of Theorem~\ref{SFTO}]
That $\phi^r$ is always surjective follows from the first fundamental theorem. Part~(i) follows from Proposition~\ref{propDualO}(i) and Theorem~\ref{ThmO}(ii).

Fix $j\in\mZ_{>0}$. Consider two partitions~$\lambda,\mu$ of elements in~$\{\mr_j-2i\,|\, 0\le i\le \mr_j/2\}$. By Proposition~\ref{propDualO}(i) and (ii), we have $\cJ_j(R(\lambda),R(\mu))=0$ unless $|\lambda|=|\mu|=\mr_j$ and~$\lambda$ and~$\mu$ are $(\mm_j+1)\times(2{\mn_j+2})$-dual. In the latter case, the space is one dimensional. If~$|\lambda|=|\mu|=\mr_j$, we have $\cC(R(\lambda),R(\mu))=e_\mu B_{\mr_j}(\delta)e_\lambda$. We can use this to define 
$$F\,=\, \sum_{\lambda\vdash\mr_j\,,\,\lambda\subset\nu^{(j)}} f_\lambda\;\;\in\; B_{\mr_j}(\delta),$$
where $f_\lambda$ is a non-zero element of~$\cJ_j(R(\lambda),R(\mu))$, for~$\mu$ the unique $(\mm_j+1)\times(2{\mn_j+2})$-dual to~$\lambda$.
By the above, we have for every~$\lambda\vdash\mr_j$  that
$$\mk\,F\circ e_\lambda\;=\;\bigoplus_{\mu,|\mu|\le\mr_j}\cJ_j(R(\lambda), R(\mu)). $$
Hence, it follows easily that the element $F$ generates~$\cJ_j(\mr_j,\mr_j)\subset B_{\mr_j}(\delta)$ as a two-sided ideal. That~$F\circ d=0$, resp $d\circ F=0$, for any Brauer diagram~$d$ which contains a cup, resp. cap follows from Proposition~\ref{propDualO}(i). Since $F$ generates~$\cJ_j(\mr_j,\mr_j)$ as a two-sided ideal, it follows that~$f_1\circ (F\otimes I^{\otimes p})\circ f_2$, for diagrams $f_1\in \cC_0(\mr_j+p,\mr_j)$, $f_2\in\cC_0(\mr_j,\mr_j+p)$ and~$p\in2\mN$ is equal to an element of the form $g_1\circ F\circ g_2$, for~$g_1,g_2\in \cC_0(\mr_j,\mr_j)$.

Since the tensor ideals in~$\cC$ form one chain, see Theorem~\ref{ThmO}, it follows that~$\cJ_j$ is generated, as a tensor ideal, by any morphism which is in~$\cJ_j$ but not in~$\cJ_{j+1}$. In particular, $F$ generates~$\cJ_j$ as a tensor ideal. For an arbitrary~$r>\mr_j$, the ideal $\cJ_j(r,r)\subset B_{r}(\delta)$ is thus spanned by elements $d_1\circ (F\otimes I^{\otimes k})\circ d_2$, for diagrams $d_1\in \cC_0(\mr_j+k,r)$, $d_2\in\cC_0(r,\mr_j+k)$, with~$k\in\mN$. In order to prove part~(ii), we need to show that it suffices to take such elements with~$k\le r-\mr_j$. If~$k>r-\mr_j$, there will be a cap in~$d_1$ and a cup in~$d_2$. Hence we can write~$d_1=d_1'\circ d_1''$, where $d_1''$ is a $(\mr_j+k,\mr_j+k-2)$-Brauer diagram consisting of one cap and~$\mr_j+k-2$ propagating lines. We similarly decompose $d_2=d_2''\circ d_2'$, where $d_2''$ consists of one cup and some propagating lines. The observation at the end of the previous paragraph implies that~$d_1''\circ (F\otimes I^{\otimes k})\circ d_2''=c_1\circ (F\otimes I^{\otimes k-2})\circ c_2$, for some~$c_1\in \cC_0(\mr_j+k-2,r)$, $c_2\in\cC_0(r,\mr_j+k-2)$. Iterating this procedure proves part~(ii).

For part~(iii), we start with the case $(m,n)=(\mm_1,\mn_1)$, so~$r_c=\mr_1$. 
For~$\lambda,\mu\vdash \mr_1$, Proposition~\ref{propDualO}(iii)  shows that~$\cJ_1(R(\lambda),R(\mu))\not=0$ implies that~$\lambda=\mu$ and~$R(\lambda)\in\sI_1$. In this case, the element $F$ is thus a summation over mutually orthogonal primitive idempotents and thus an idempotent. The fact that~$F\circ d=0=d\circ F$, for any~$d\in B_{\mr_1}(\delta)$ contained in the ideal spanned by diagrams with cups and caps means that taking an appropriate sum over conjugate idempotents yields a central idempotent.

Now assume that~$j>1$. By Proposition~\ref{propDualO}(iii), any element $F$ as in part~(ii) will, up to conjugacy, contain a term $f:R(\lambda)\to R(\mu)$, with neither $R(\lambda)$ nor~$R(\mu)$ in~$\sI_j$ and $\lambda\not=\mu$. Furthermore, Proposition~\ref{propDualO}(ii) implies~$e_\mu F=Fe_\lambda=f$. Hence, we have 
$$e_\mu F^2e_\lambda=f^2=0\not=f=e_\mu Fe_\lambda$$ and $F$ cannot be an idempotent.
\end{proof}

\subsection{The general linear case}

\begin{thm}\label{SFTG} For an $\mF_2$-graded vector space~$V$ of dimension~$(m,n)$ with~$\delta=m-n$, and dual space $W:=V^\ast$, we set~$r_c=(m+1)(n+1)$ and consider the surjective algebra morphism
$$\phi^{k,l}:\;B_{k,l}(\delta)\;\tto\;\End_{\GL(V)}(V^{\otimes k}\otimes W^{\otimes l}),\quad\mbox{for~$k,l\in\mN$}.$$
\begin{enumerate}[(i)]
\item If~$k+l<r_c$, then $\phi^{k,l}$ is an isomorphism.
\item For all~$a,b\in\mN$ with~$a+b=r_c$, there exists an element $F_{a,b}\in B_{a,b}(\delta)$, such that~$F\circ d=0=d\circ F$, for any diagram~$d\in B_{a,b}(\delta)$ which contains a cup. Furthermore, for all~$k+l\ge r_c$, the kernel of~$\phi^{k,l}$ is generated as a two-sided ideal in~$B_{k,l}(\delta)$ by the element $I^{\otimes k-a}\otimes F_{a,b}\otimes I^{l-b}$, 
for arbitrary~$a,b$ with~$a\le k$, $b\le l$ and $a+b=r_c$.
 \item The element $F_{a,b}$ can be chosen to be an idempotent if and only if $mn=0$. Under those assumptions, $F_{a,b}$ can also be chosen to be primitive and central in~$B_{a,b}(\delta)$. 
\end{enumerate}
\end{thm}

\begin{rem}${}$
\begin{enumerate}[(i)]
\item Theorem~\ref{SFTG}(i) is well-known, see e.g. \cite[Corollary~2.4]{LZ-SFT}.
\item The special case $l=0$ of Theorem~\ref{SFTG}(ii) is implied by \cite[Theorem~2.3]{LZ-SFT}. The case $l>0$ seems to be new.
\end{enumerate}
\end{rem}

We start the proof with the following proposition.

\begin{prop}\label{propDualG}
Consider $\lambda^\bullet,\lambda^\circ,\mu^\bullet,\mu^\circ\in\Par$, $j\in\mN$ and recall~$\mr_j,\mm_j,\mn_j$ from \ref{DefmnG}
\begin{enumerate}[(i)]
\item For~$a,b,c,d\in \mN=\Ob\cC_0$ with~$a+d<\mr_j$ or~$b+c<\mr_j$, we have $\cJ_j([a,b],[c,d])=0$. Hence 
$$\cJ_j(R(\lambda^\bullet,\lambda^\circ),R(\mu^\bullet,\mu^\circ))=0,\quad\mbox{if $|\lambda^\bullet|+|\mu^\circ|<\mr_j$ or~$|\mu^\bullet|+|\lambda^\circ|<\mr_j$.}$$
\item If~$|\lambda^\bullet|+|\mu^\bullet|+|\lambda^\circ|+|\mu^\circ|=2\mr_j$, we have
$$\dim_{\mk}\cJ_j((R(\lambda^\bullet,\lambda^\circ),R(\mu^\bullet,\mu^\circ))\;=\;\begin{cases}1 &\mbox{if $\lambda^\bullet$ and~$\mu^\circ$, as well as $\lambda^\circ$ and~$\mu^\bullet$},\\
&\mbox{are $(\mm_j+1)\times({\mn_j+1})$-dual}\\
0&\mbox{otherwise}.
\end{cases}$$

\item Fix arbitrary $a,b\in\mN$ with $a+b=\mr_j$ and assume $|\lambda^\bullet|=|\mu^\bullet|=a$ and $|\lambda^\circ|=|\mu^\circ|=b$.
\begin{enumerate}[(a)]
\item If~$j=1$, we have $\cJ_1(R(\lambda^\bullet,\lambda^\circ),R(\mu^\bullet,\mu^\circ))=0$ for all but one choice of $\lambda^\bullet,\lambda^\circ,\mu^\bullet,\mu^\circ$. That choice satisfies $[\lambda^\bullet,\lambda^\circ]=[\mu^\bullet,\mu^\circ]$ and~$R([\lambda^\bullet,\lambda^\circ])\in \sI_1$.
\item If~$j>1$, there exist such $\lambda^\bullet,\lambda^\circ,\mu^\bullet,\mu^\circ$ for which~$\cJ_j(R(\lambda^\bullet,\lambda^\circ),R(\mu^\bullet,\mu^\circ))\not=0$ but $\sI_j\not\ni R(\lambda^\bullet,\lambda^\circ)\not= R(\mu^\bullet,\mu^\circ)\not\in\sI_j$.
\end{enumerate}

\end{enumerate}
\end{prop}
\begin{proof}
We will freely use equation~\eqref{RvsnG}. Theorem~\ref{ThmG}(iii) implies that
$\cJ_j(0,[k,l])=0$, for all~$[k,l]\in \Ob\cC_0$ with~$k<\mr_j$ or~$l<\mr_j$.
Part~(i) then follows from Lemma~\ref{LemJJ}, since
$$\cJ_{j}([a,b],[c,d])\simeq\cJ_j(\unit,[c+b,d+a])=0.$$

Lemma~\ref{LemJJ} also implies
$$\cJ_j(R(\lambda^\bullet,\lambda^\circ),R(\mu^\bullet,\mu^\circ))\simeq \cJ_j(\unit, R(\lambda^\circ,\lambda^\bullet)\otimes R(\mu^\bullet,\mu^\circ)).$$ By Theorem~\ref{ThmG}(iii), we thus find that~$\dim_{\mk}\cJ_j(R(\lambda^\bullet,\lambda^\circ),R(\mu^\bullet,\mu^\circ))$ is equal to the number of times that~$R(\nu^{(j)\bullet},\nu^{(j)\circ})$ occurs as a direct summand in~$R(\lambda^\circ,\lambda^\bullet)\otimes R(\mu^\bullet,\mu^\circ)$. Part~(ii) thus follows from Lemmata~\ref{LemSC}(i) and~\ref{LemRect}.

The proof of part (iii) is analogous to the proof of Proposition~\ref{propDualO}(iii).\end{proof}

\begin{proof}[Proof of Theorem~\ref{SFTG}]
That $\phi^{k,l}$ is always surjective follows from the first fundamental theorem. Part~(i) follows from Proposition~\ref{propDualG}(i) and Theorem~\ref{ThmG}(ii).

Fix $j\in\mN$ and arbitrary~$a,b\in\mN$ with~$a+b=\mr_j$. Consider $i_1,i_2\le\min(a,b)$ and partitions~$\lambda^\bullet\vdash a-i_1$, $\lambda^\circ\vdash b-i_1$, $\mu^\bullet\vdash a-i_2$ and $\mu^\circ\vdash b-i_2$. By Proposition~\ref{propDualG}(i) and (ii), we have $\cJ_j(R(\lambda^\bullet,\lambda^\circ),R(\mu^\bullet,\mu^\circ))=0$ unless $i_1=0=i_2$ and $[\lambda^\bullet,\lambda^\circ]$ and $[\mu^\bullet,\mu^\circ]$ are `dual' in the appropriate sense. In the latter case, the space is one dimensional. Hence, we can define, as in the proof of Theorem~\ref{SFTO}, an element $F_{a,b}\in B_{a,b}(\delta)$, which generates~$\cJ_j([a,b],[a,b])\subset B_{a,b}(\delta)$ as a two-sided ideal. That~$F_{a,b}\circ d=0$, resp. $d\circ F_{a,b}=0$, for any walled Brauer diagram~$d$ which contains a cup, resp. cap follows from Proposition~\ref{propDualG}(i). Since $F_{a,b}$ generates~$\cJ_j([a,b],[a,b])$ as a two-sided ideal, it follows that~$f_1\circ (I^{\otimes p}\otimes F_{a,b}\otimes I^{\otimes p})\circ f_2$, for diagrams $f_1\in \cC_0([a+p,b+p],[a,b])$, $f_2\in\cC_0([a,b],[a+p,b+p])$ and~$p\in\mN$, is equal to an element of the form $g_1\circ F_{a,b}\circ g_2$, for certain~$g_1,g_2\in \cC_0([a,b],[a,b])$.

Since the tensor ideals in~$\cC$ form one chain, see Theorem~\ref{ThmG}, it follows that~$\cJ_j$ is generated, as a tensor ideal, by any morphism which is in~$\cJ_j$ but not in~$\cJ_{j+1}$. In particular, $F_{a,b}$ for arbitrary~$a,b$ with~$a+b=\mr_j$ generates~$\cJ_j$ as a tensor ideal. Take arbitrary~$k+l>\mr_j$ and assume $a\le k$ and $b\le l$. The ideal $\cJ_j([k,l],[k,l])\subset B_{k,l}(\delta)$ is thus spanned by elements $d_1\circ (I^{s_1}\otimes F_{a,b}\otimes I^{s_2})\circ d_2$, for diagrams $d_1\in \cC_0([a+s_1,b+s_2],[k,l])$, $d_2\in\cC_0([k,l],[a+s_1,b+s_2])$, with~$s_1,s_2\in\mN$ such that~$s_1-s_2=k-l-a+b$. In order to prove part~(ii), we need to show that it suffices to take such elements with~$s_1= k-a$. If~$s_1>k-a$, there will be a cap in~$d_1$ and a cup in~$d_2$. Hence we can write~$d_1=d_1'\circ d_1''$, where $d_1''$ is a $([a+s_1,b+s_2],[a+s_1-1,b+s_2-1])$-diagram consisting of one cap and otherwise only propagating lines. We similarly decompose $d_2=d_2''\circ d_2'$, where $d_2''$ consists of one cup and some propagating lines. The observation at the end of the previous paragraph implies that~$d_1''\circ (I^{\otimes s_1}\otimes F_{a,b}\otimes I^{\otimes s_2})\circ d_2''=c_1\circ (I^{\otimes s_1-1}\otimes F\otimes I^{\otimes s_2-1})\circ c_2$, for some~$c_1\in \cC_0([a+s_1-1,b+s_2-1],[k,l])$, $c_2\in\cC_0([k,l],[a+s_1-1,b+s_2-1])$. Iterating this procedure proves part~(ii).

For part~(iii), we start with the case $(m,n)=(\mm_1,\mn_1)$, in the notation of~\ref{DefmnG}, so~$r_c=\mr_1$. 
Proposition~\ref{propDualG}(iii)(a)  shows the element $F=F_{a,b}$ is a primitive idempotent. Since the corresponding Young diagrams are either one column or one row, the primitive idempotent is central in $\mk (\SG_a\times \SG_b)$. The fact that~$F\circ d=0=d\circ F$, for any~$d\in B_{a,b}(\delta)$ contained in the ideal spanned by diagrams with cups and caps means that $F$ will also be central in $B_{a,b}$. Part (iii)(b) follows as in Theorem~\ref{SFTO}.\end{proof}

\subsection{The periplectic case}

\begin{thm}\label{SFTP} For an $\mF_2$-graded vector space~$V$ of dimension~$(n,n)$, we set~$r_c=\frac{1}{2}(n+1)(n+2)$ and consider the surjective algebra morphism
$$\phi^r:\;A_r\;\tto\;\End_{\Pe(V)}(V^{\otimes r}),\qquad\mbox{for~$r\in\mN$}.$$
\begin{enumerate}[(i)]
\item If~$r<r_c$, then $\phi^r$ is an isomorphism.
\item If~$r\ge r_c$, the kernel of~$\phi^r$ is generated as a two-sided ideal by a single element $F\otimes I^{r-r_c}$, with~$F\in A_{r_c}$, such that~$F\circ d=0$, resp. $d\circ F=0$, for any Brauer diagram~$d$ which contains a cup, resp. cap. Moreover, $F$ is never an idempotent if $n>0$.
\end{enumerate}
\end{thm}
\begin{proof}
Mutatis mutandis the proof of Theorem~\ref{SFTO} or~\ref{SFTG}.
\end{proof}

\begin{rem}
Theorem~\ref{SFTP}(i) for the special case $r\le n$ is in~\cite[Theorems~4.1 and~4.5]{Moon}.
\end{rem}

\begin{prop}\label{propDualP}
Consider $\lambda,\mu\in\Par$ and~$j\in\mN$.
\begin{enumerate}[(i)]
\item For~$a,b\in \mN=\Ob\cC_0$ with~$a+b<j(j+1)$, we have $\cJ_j(a,b)=0$. Hence we have 
$$\cJ_j(R(\lambda),R(\mu))=0,\quad\mbox{if $|\lambda|+|\mu|<j(j+1)$.}$$
\item If~$|\lambda|+|\mu|=j(j+1)$, we have
$$\dim_{\mk}\cJ_j(R(\lambda),R(\mu))\;=\;\begin{cases}1 &\mbox{if $\lambda^t$ and~$\mu$ are $j\times(j+1)$-dual,}\\
0&\mbox{otherwise}.
\end{cases}$$
\end{enumerate}
\end{prop}
\begin{proof}
Mutatis mutandis Proposition~\ref{propDualO} or~\ref{propDualG}.
\end{proof}

\section{Further applications}\label{SecFurther}

\subsection{Invariant theory for the symmetric group}
For Deligne's category $\RS$ as in Section~\ref{SecRS}, we have the following theorem. Part~(i) can also be obtained from results of Comes - Ostrik in~\cite{ComesOst}. Part (ii) is originally due to Jones in~\cite{Jones2}. Part (iii) provides an alternative proof of a special case of a recent result of Benkart - Halverson in~\cite[Theorem~5.6]{BH}. The relevant idempotent has even been explicitly constructed diagrammatically {\it loc. cit}.
\begin{thm}Set $\cC:=\RS$, with~$t\in\mZ_{>0}$.
\begin{enumerate}[(i)]
\item The unique proper tensor ideal $\cN$ in $\cC$ satisfies
$$\cN(a,b)=0,\qquad\mbox{for all~$a,b\in\mN=\Ob\cC_0$, with~$a+b\le t$.}$$
\item With $M_t$ the permutation module of~$\SG_t$ and $P_k(t)=\cC_0(k,k)$ the partition algebra, the surjective algebra morphism obtained from the functor $\Fmod$ in Corollary~\ref{CorSt}
$$\phi^k:\;P_k(t)\;\tto\;\End_{\SG_t}(M_t^{\otimes k}),\quad \mbox{for $k\in\mN,$}$$
is an isomorphism when $2k\le t$.
\item Assume that~$t$ is odd. There exists a central primitive idempotent $E\in P_{\frac{t+1}{2}}(t)$, such that the kernel of~$\phi^k$ is generated as a two-sided ideal by $E\otimes I^{\otimes(k-\frac{t+1}{2})}$, for~$k>t/2$.
\end{enumerate}
\end{thm}
\begin{proof}
Part (i) follows immediately from the proof of Proposition~\ref{PropCO}, using Lemma~\ref{LemJJ}. Part~(ii) is the special case $a=b=k$ of part (i).

The indecomposable objects in $\cC$ can be labelled as $R(\lambda)$, with~$\lambda\in\Par$. Furthermore, $R(\lambda)\inplus k$, with~$k\in\mN=\Ob\cC_0$, if and only if $|\lambda|\le k$, see~\cite[2.2.3]{ComesOst}. It then follows easily that we have a monoidal functor $\Tmod:\cS\to\cC$, with $\cS$ the category in~\ref{defSCO} and $\Tmod(i)=i$ for $i\in\mN=\Ob\cS_0=\Ob\cC_0$. Furthermore, properties (a) and (b) in~\ref{tuple} are satisfied, for some $\phi:\Par\to\Par$. It follows that $\cN(R(\lambda),R(\mu))=0$ unless $|\lambda|+|\mu|\ge t+1$. 

Now assume that $t$ is odd. If we have some $\lambda\vdash \frac{t+1}{2}$ for which $\cN(R(\lambda),R(\lambda))\not=0$, then the Specht module~$(t+1)$ must appear in the decomposition of $\phi(\lambda)\otimes \lambda$. By Lemma~\ref{LemRect}(ii), this forces $\lambda=\frac{t+1}{2}=\phi(\frac{t+1}{2})$. By \cite[Section~3.5]{ComesOst}, $R((\frac{t+1}{2}))$ is contained in the unique proper tensor Ob-ideal, which thus implies that indeed $(\frac{t+1}{2})=\phi(\frac{t+1}{2})$. Set $\lambda^0:=(\frac{t+1}{2})$. By \cite[Theorem~2.6]{ComesOst}, we have that $\cC(R(\lambda^0),R(\lambda^0))$ is one dimensional. 
 The identity morphism of~$R(\lambda^0)$ then yields the idempotent in part (iii), as in the proof of Theorem~\ref{SFTG}(iii).
\end{proof}

\subsection{Quantum Deligne categories}
In this section, we take $\mk=\mC$. It is expected that quantised versions of Deligne's categories are generically equivalent to Deligne categories, as $\mC$-linear categories, but not necessarily as monoidal categories. These thus provide examples where Corollary~\ref{Cor4Bru} might be applied. The equivalence between $\underline{\Rep}U_q(\mathfrak{gl}_\delta)$, the quasi-abelian envelope of the ``Hecke category'' of~\cite[\S 5.2]{Turaev}, and $\RG$ was recently proved by Brundan in~\cite{Brundan}. We find the following consequence. The indecomposable objects in $\underline{\Rep}U_q(\mathfrak{gl}_\delta)$ can again be labelled by bipartitions as $R(\lambda^\bullet,\lambda^\circ)$.

\begin{thm}\label{ThmGq}
Take  $q\in\mC$ not a root of unity and $\delta\in\mZ$.
All tensor ideals in~$\underline{\Rep}U_q(\mathfrak{gl}_\delta)$ are in the set~$\{\cJ_i\,|\,i\in\mN\}$ and form one chain
$$\underline{\Rep}U_q(\mathfrak{gl}_\delta)\;=\;\cJ_0\;\supsetneq\;\cJ_1\;\supsetneq\cJ_2\;\supsetneq\; \cJ_3\;\supsetneq\; \cdots.$$
For~$j\in\mZ_{>0}$ and objects $X,Y$, the $\mC$-module~$\cJ_j(X,Y)$ consists of all morphisms which factor as $X\to Z\to Y$, with~$Z$ a direct sum of~$R(\lambda^\bullet,\lambda^\circ)$ with $[\lambda^\bullet,\lambda^\circ]$ as in~\eqref{eqhookG}.
\end{thm}
\begin{proof}
We set $\cD=\underline{\Rep}U_q(\mathfrak{gl}_\delta)$ and we will use Theorem~\ref{ThmG} for~$\cC=\RG$ freely. The equivalence $\cC\stackrel{\sim}{\to}\cD $ is given in~\cite[Corollary~1.12]{Brundan}. As explained in the paragraph following \cite[Corollary~1.12]{Brundan}, this equivalence induces an isomorphism of split Grothendieck {\em rings}, mapping $[R(\lambda^\bullet,\lambda^\circ)]$ to $[R(\lambda^\bullet,\lambda^\circ)]$. In particular, the thick ideals in the split Grothendieck rings for both categories are identical. Hence, the poset $\TId([\cD]_{\oplus})$ is one chain $\sI_0\supsetneq \sI_1\supsetneq \sI_2\supsetneq\cdots$, with~$R(\lambda^\bullet,\lambda^\circ)\in\sI_j$ if and only if the condition in~\eqref{eqhookG} on $[\lambda^\bullet,\lambda^\circ]$ is satisfied. By Corollary~\ref{Cor4Bru}, we also find that the tensor ideals in 
$\cD$ form one chain $\cJ_0\supsetneq \cJ_1\supsetneq \cJ_2\supsetneq\cdots$.

Recall the bipartitions $[\nu^{(j)\bullet},\nu^{(j)\circ}]$ from Lemma~\ref{Lem1G}. By Theorem~\ref{ThmSC}(iii) and Corollary~\ref{Cor4Bru}, we have 
$$\cJ_i(\unit,R(\nu^{(j)\bullet},\nu^{(j)\circ}))\;=\;\begin{cases}\cD(\unit,R(\nu^{(j)\bullet},\nu^{(j)\circ}))&\mbox{if $i\le j$}\\
0&\mbox{otherwise.}\end{cases}$$ Since we have $R(\nu^{(j)\bullet},\nu^{(j)\circ})\in\sI_j$, it follows that~$\cJ^{\min}_{\sI_j}= \cJ_{k}$, for some~$k\le j$. By iteration on $j$, we therefore find that~$\cJ^{\min}_{\sI_j}= \cJ_j$, which concludes the proof.
\end{proof}

When $q\in\mC$ is a root of unity (other than $\pm1$) it follows from Remark~\ref{Rem3}, that $\Ob:\TId(\underline{\Rep}U_q(\mathfrak{gl}_\delta))\tto \Ide([\underline{\Rep}U_q(\mathfrak{gl}_\delta)]_{\oplus})$ will generally not be a bijection.

\subsection{Supergroups of type $Q$}

The class of classical algebraic affine supergroup schemes contains, along with some exceptional supergroups, also the supergroups of type $Q$, see e.g.~\cite[Section~2]{Vera}. In~\cite{ComesKuj}, the {\em oriented Brauer-Clifford category}  $\mathcal{OBC}$ was introduced. This is a strict monoidal supercategory with a full monoidal superfunctor
$$\mathcal{OBC}\;\to \;\Rep_{\mk} \mathrm{Q}(n),\quad\mbox{for all~$n\in\mN$,}$$
 see \cite[Section~4.2]{ComesKuj} and~\cite{Sergeev}. It is natural to define $\RQ$ as $\mathcal{OBC}^{\oplus\sharp}$.

\begin{conj}
The tensor ideals in~$\mathcal{OBC}$ are precisely the kernels of the superfunctors~$\mathcal{OBC}\to \Rep_{\mk} \mathrm{Q}(n).$ These are in natural bijection with the thick tensor~$\Ob$-ideals in~$\RQ$.
\end{conj}

\subsection{Modular versions} We expect the classification of tensor ideals in Deligne categories over (algebraically closed) fields $\mk$ of characteristic $p>0$ to be significantly more difficult. Observe that  the image of $\underline{\Rep}GL_n\to \Rep_{\mk}\GL_n$ is contained in $\Ti(\GL_n)$ and the functor $\underline{\Rep}GL_n\to \Ti(\GL_n)$ is dense if $p>n=h$ by \cite[Lemma~3.4]{Donkin}. We conclude the following.

\begin{enumerate}[(a)] 
\item In characteristic zero, we found that each functor~$\RG\to \Rep_{\mk}\GL(m|n)$ `contributed' precisely one tensor ideal. By Proposition~\ref{ModularProp}, one might expect that each such functor contributes infinitely many tensor ideals when $\charr(\mk)>0$. Contrary to Proposition~\ref{ModularProp}, in characteristic zero, the category $\Ti(\GL_n)\simeq \Rep_{\mk}\GL_n$ has no proper tensor ideals.
\item By Proposition~\ref{ObstGLm}, the tensor ideals in $\RG$ will generally not be in bijection with thick ideals in the Grothendieck ring $[\RG]_{\oplus}$, when $\charr(\mk)>0$.
\end{enumerate}


\appendix

\section{Monoidal supercategories}\label{SecSuper}  We will consider gradings by the group $\mF_2^+:=\mZ/2\mZ=\{\oa,\ob\}$.

\subsection{Definitions} 
\subsubsection{} For an $\mF_2$-graded ring $R$, we consider the category~$R$-sMod, which is the category of~$\mF_2$-graded modules, with all~$R$-linear morphisms. By $R$-gMod we denote the subcategory with same objects but where the morphisms have to preserve the $\mF_2$-grading.
Consider an $\mF_2$-graded $R$-module~$M$. For~$v\in M_{\oa}$, resp. $v\in M_{\ob}$, we write~$|v|=0$, resp. $|v|=1$. For~$v\in M$, we also write~$v^{\oa},v^{\ob}$, for the unique $v^i\in M_i$, such that~$v=v^{\oa}+v^{\ob}$.

\subsubsection{} A {\bf supercategory} is a category enriched over the monoidal category $\mZ$-gMod. In particular, supercategories are preadditive. An example of a supercategory is $R$-sMod, for an $\mF_2$-graded ring $R$. A {\bf superfunctor} between two supercategories is a functor enriched over~$\mZ$-gMod. For~$\cC$ a supercategory and $X\in\Ob\cC$, the functor
$$\Pmod_X^{\cC}:=\cC(X,-):\; \cC\to \mZ\mbox{-sMod}$$
is a superfunctor.
For two supercategories~$\cC_1,\cC_2$ and superfunctors~$\Fmod,\Gmod :\cC_1\to\cC_2$, an {\em even} natural transformation~$\xi:\Fmod\Rightarrow \Gmod $ is one in which every morphism is even. 
For a supercategory~$\cC$, the ring $\mZ[\cC]$ has the structure of an $\mF_2$-graded ring.

\subsubsection{}\label{ABCsuperIdeals} An ideal $\cJ$ in a supercategory~$\cC$ is an ideal as in~\ref{IdPA}, with the extra assumption that~$\cJ(X,Y)$ is a {\em graded} subgroup of $\cC(X,Y)$, for all~$X,Y\in\Ob\cC.$

Denote by~$\cC$-gMod the category of superfunctors from $\cC$ to~$\mZ$-sMod, with morphisms given by even natural transformations.
We have an equivalence between $\cC$-gMod and~$\mZ[\cC]$-gMod. We denote the partially ordered set of {\em graded} submodules of~$\Pmod_{X}^{\cC}$ by~$\Sub_{\cC}(\Pmod^{\cC}_X)$.
 
\subsubsection{} For supercategories~$\cC$ and~$\cD$, the product $\cC\boxtimes\cD$ is the supercategory with same objects as $\cC\times\cD$ and graded morphism groups given by 
$$\cC\boxtimes\cD((X,Y),(W,Z))\;=\;\cC(X,W)\otimes_{\mZ}\cD(Y,Z).$$
Composition of morphisms is defined by 
$$(f\boxtimes g)\circ (h\boxtimes k)\;=\;(-1)^{|g||h|}(f\circ h\boxtimes g\circ k).$$

\begin{rem}
In the above {\em super interchange law}, the morphisms~$g,h$ are assumed to be homogeneous, with respect to the $\mF_2$-grading. Expressions like this determine the general rule by additivity. We will keep this convention throughout the appendix.
\end{rem}

\subsubsection{} A {\bf strict monoidal supercategory} $\cC$ is a triple~$(\cC,\otimes,\unit)$, with~$\cC$ a supercategory, a superfunctor~$\otimes:\cC\boxtimes\cC\to\cC$, and an object $\unit\in\Ob\cC$, satisfying the same relations as for a strict monoidal category. Going to non-strict monoidal supercategories corresponds again to relaxing the equalities of functors to {\em even} natural isomorphisms. Because all the isomorphisms are even, the coherence conditions do not change. The super interchange law however implies that~$\cC(\unit,\unit)$ is now super commutative.

The notion of left-tensor ideals extends immediately to monoidal supercategories, taking into account the restricted notion of ideals in~\ref{ABCsuperIdeals}. The corresponding partially ordered set is again denoted by~$\TId(\cC)$.
The notion of thick left-tensor~$\Ob$-ideals does not change compared to monoidal categories, as the split Grothendieck ring $[\cC]_{\oplus}$ does not inherit an $\mF_2$-grading. The corresponding partially ordered set is again denoted by~$\Ide([\cC]_{\oplus})$.

\subsection{Duals in monoidal supercategories}
Fix a strict monoidal supercategory~$\cC$.

\subsubsection{}\label{DefDualS}A {\bf right dual} of~$X\in\Ob\cC$ is a triple~$(X^\vee,\ev_X,\co_X)$ with~$X^\vee\in \Ob\cC$ and morphisms
$$\ev_X:X\otimes X^\vee\to\unit\qquad\mbox{and}\qquad \co_X:\unit\to X^\vee\otimes X,$$
which satisfy
$$\sum_{i\in\mF_2}(\ev^i_X\otimes 1_X)\circ (1_X\otimes \co^i_X)=1_X,\qquad \sum_{i\in\mF_2}(-1)^i(1_{X^{\vee}}\otimes\ev^i_X)\circ(\co^i_X\otimes 1_{X^\vee})=1_{X^\vee},$$
$$(\ev^i_X\otimes 1_X)\circ (1_X\otimes \co^{i+\ob}_X)=0\quad\mbox{and}\quad(1_{X^{\vee}}\otimes\ev^i_X)\circ(\co^{i+\ob}_X\otimes 1_{X^\vee})=0,\quad\mbox{ for all~$i\in\mF_2$.}$$
If~$X$ and~$Y$ admit duals $X^\vee,Y^\vee$, then $Y\otimes X$ admits a dual $X^\vee\otimes Y^\vee$, with
$$\ev_{Y\otimes X}\;=\;\ev_Y\circ (1_Y\otimes\ev_X\otimes 1_{Y^\vee})\qquad\mbox{and}\qquad \co_{Y\otimes X}\;=\;\sum_{i,j\in\mF_2}(-1)^{ij} (1_{X^\vee}\otimes \co^i_Y\otimes 1_X)\circ\co^j_{X}.$$
If all objects in~$\cC$ admit a right dual, we say that~$\cC$ is {\bf right rigid}.

\subsubsection{}
When $\co_X=\co_X^{\oa}$, resp. $\co_X=\co_X^{\ob}$, we say that~$X$ has an {\bf even} resp. {\bf odd dual}. We denote the parity of such a homogeneous dual by~$d_X$. If~$X$ admits a homogeneous dual, the relations in~\ref{DefDualS} simplify to
$$(\ev\otimes 1_X)\circ (1_X\otimes _X)=1_X\qquad\mbox{and}\qquad (1_{X^{\vee}}\otimes\ev_X)\circ(\co_X\otimes 1_{X^\vee})=(-1)^{d_X}1_{X^\vee},$$

For~$X,Y\in\cC$, such that~$X$ admits a homogeneous dual of parity $d_X\in\mF_2$, equation~\eqref{eqiota} yields an isomorphism $\iota_{XY}:\cC(\unit,X^\vee\otimes Y)\stackrel{\sim}{\to} \cC(X,Y)$, of parity $|\iota_{XY}|=d_X$, with inverse
$$\iota^{-1}_{XY}(f)\;=\; (-1)^{|f|d_X}(1_{X^\vee}\otimes f)\circ\co_X,\qquad\mbox{for~$f\in \cC(X,Y)$.}$$

\subsubsection{} For~$X$ with an arbitrary right dual $X^\vee$, we can define elements of $\cC(X,X)$ as
$$a_X:=(\ev^{\oa}_X\otimes 1_X)\circ (1_X\otimes \co^{\oa}_X)\quad\mbox{and}\quad b_X:=(\ev^{\ob}_X\otimes 1_X)\circ (1_X\otimes \co^{\ob}_X).$$
By definition, we then have 
$$a_X^2=a_X, \;\;a_Xb_X=0=b_Xa_X, \;\;b_X^2=b_X\; and \;\;1_X=a_X+b_X.$$ 
Similar properties hold for~$X^\vee$. If~$\cC$ is karoubian, it then follows that~$X$ is a direct sum of an object with an even dual and one with an odd dual.

\subsection{Theorems}\label{ThmApp}

\begin{thm}\label{super1}
For a right rigid monoidal supercategory~$\cC$, the assignment
$$\Psi:\TId(\cC)\to \Sub_{\cC}(\Pmod_{\unit}),\quad \cJ\mapsto\cJ(\unit,-),$$
yields an isomorphism of partially ordered sets.
\end{thm}
Although not essential, we can prove the theorem for~$\cC^{\sharp}$, instead of~$\cC$, which means that we can assume that every object is a finite direct sum of objects with either even or odd right dual. It then suffices to work with objects $X$ which admit either an even or odd dual. Theorem~\ref{super1} now follows from the following two lemmata. 
\begin{lemma}
Consider homogeneous~$\phi\in\cC(\unit,X^\vee\otimes Y)$ for~$X,Y\in\Ob\cC$.
\begin{enumerate}[(i)]
\item For~$Z\in \Ob\cC$ and~$g\in\cC(Y,Z)$, we have
$$g\circ\iota_{XY}(\phi)\;=\;(-1)^{|g|d_X}\iota_{XZ}((1_{X^\vee}\otimes g)\circ\phi).$$
\item For~$W\in\Ob\cC$ and~$f\in\cC(W,X)$, we have
$$\iota_{XY}(\phi)\circ f\;=\;\iota_{WY}((\chi\otimes 1_Y)\circ \phi),\quad\mbox{with }\;\chi:=(-1)^{|f||\phi|+d_Xd_W}(1_{W^\vee}\otimes\ev_X)\circ (\iota^{-1}_{WX}(f)\otimes 1_{X^\vee}).$$
\end{enumerate}
\end{lemma}

\begin{lemma}
Consider $\phi\in\cC(\unit,X^\vee\otimes Y)$ for~$X,Y\in\Ob\cC$. For~$Z\in\Ob\cC$, we have
$$1_Z\otimes \iota_{XY}(\phi)=\iota_{Z\otimes X,Z\otimes Y}(\psi),\quad\mbox{with~$\psi=(-1)^{d_Xd_W}(1_{X^\vee}\otimes \co_Z\otimes 1_Y)\circ \phi.$}$$
\end{lemma}

Similarly, the proofs of the following theorems do not change substantially from the ones of Theorems~\ref{ThmFibre} and~\ref{ThmSuff}.
\begin{thm}
Let~$\cC$ be a right rigid Krull-Schmidt monoidal supercategory.
\begin{enumerate}[(i)]
\item We have a surjective morphism of partially ordered sets
$$\Ob:\TId(\cC)\tto \Ide([\cC]_{\oplus}).$$
\item 
For~$\sI\in \Ide([\cC]_{\oplus})$, the minimal element in the fibre $\Ob^{-1}(\sI)$ is given by the tensor ideal
$$\cJ^{\min}_{\sI}(X,Y)\;=\;\{f\in\cC(X,Y)\,|\, \mbox{there exists~$Z\in\sI$ such that~$f$ factors as $X\to Z\to Y$}\}.$$
\item The minimal element in~$\Psi\circ\Ob^{-1}(\sI)$ is given by~$\Tr_{\sI}\Pmod_{\unit}$.
\end{enumerate}
\end{thm}

\begin{thm}
Consider a right rigid Krull-Schmidt monoidal supercategory~$\cC$. Assume that for each $Z\in \sB$ (with~$\sB$ as in~\ref{SetB})
\begin{enumerate}[(a)]
\item the $\mF_2$-graded $\cC(Z,Z)$-module~$\cC(\unit,Z)$ has no proper graded submodules.
\item there exists~$X_Z\in\inde\cC$, such that  $\add(X_Z^\vee\otimes X_Z)\cap\sB=\{Z\}.$
\end{enumerate}
 Then $\Ob(-): \TId(\cC)\to \Ide([\cC]_\oplus)$ is an isomorphism of partially ordered sets.
\end{thm}

\subsection*{Acknowledgement}
The research was supported by ARC grant DE170100623. The author thanks Jon Brundan, Michael Ehrig, Inna Entova, Jim Humphreys, Gus Lehrer, Andrew Mathas, Geordie Williamson 
and Yang Zhang for interesting discussions.

\begin{flushleft}
	K. Coulembier\qquad \url{kevin.coulembier@sydney.edu.au}
	
	School of Mathematics and Statistics, University of Sydney, NSW 2006, Australia
	
	\end{flushleft}

\end{document}